\documentclass{article}
\usepackage{ulem,hhline,dsfont,pifont,multicol,lmodern}
\usepackage{amscd}
\usepackage{lipsum}
\usepackage[all,cmtip]{xy}
\usepackage{amsmath,amsthm,amsfonts,amsbsy,amssymb,geometry,times,pst-node}
\usepackage{mathrsfs} 
\usepackage{latexsym}
\usepackage{mathtools}
\usepackage{enumitem} 
\usepackage[utf8]{inputenc}
\usepackage[english]{babel}
\usepackage{ mathrsfs }

\usepackage{graphicx}

\newcommand{\Z}{\mathbb{Z}}

\usepackage[all]{xy}

\def\A{{\mathscr{A}}}
\def\Mod{{\rm Mod}}

\def\Ext{{\rm Ext}}

\def\Hom{{\rm Hom}}
\def\Im{{\rm Im}}
\def\Ker{{\rm Ker}}
\def\Coker{{\rm Coker}}

\def\wdim{{\rm w.gl.dim}}
\def\F{{\cal F}}
\def\FP{{\cal F}^{\perp}}

\newtheorem{thm}{\bf Theorem}[section]
\newtheorem{cor}[thm]{\bf Corollary}

\newtheorem{prop}[thm]{\bf Proposition}
\newtheorem{Def}[thm]{\bf Definition}

\newtheorem{ex}[thm]{\bf Example}

\begin{document}
	\title{Subprojectivity   in abelian categories}
	\author{Houda Amzil, Driss  Bennis, J. R. Garc\'{\i}a Rozas,\\
	 Hanane Ouberka and Luis Oyonarte}

	\date{}
	
	\maketitle

\bigskip

\noindent{\large\bf Abstract.} In the last few years, López-Permouth and several collaborators have introduced a new approach in the study of the classical projectivity, injectivity and flatness of modules. This way, they introduced subprojectivity domains of modules as a tool to measure, somehow, the projectivity level of such a module (so not just to determine whether or not the module is projective). In this paper we develop a new treatment of the subprojectivity in any abelian category which shed more light on some of its various important aspects.  Namely, in terms of subprojectivity, some classical results are unified and some classical rings are characterized. It is also shown that, in some categories, the subprojectivity measures notions other than the projectivity. Furthermore, this new approach allows, in addition to   establishing nice generalizations of known results, to construct various new examples such as the subprojectivity domain of the class of Gorenstein projective objects, the class of semi-projective complexes and particular types of representations of a finite linear quiver. The paper ends with a study showing that the fact that a subprojectivity domain of a class coincides with its first right Ext-orthogonal class can be characterized in terms of the existence of preenvelopes and precovers.


\bigskip





\section{Introduction}
Throughout this paper, $\mathscr{A}$ will denote an abelian category with enough projectives. We denote the class of projective objects by $Proj(\mathscr{A})$ and the class of injective objects by $Inj(\mathscr A)$. Also, $R$ will denote an associative ring with identity and modules will be unital left $R$-modules, unless otherwise explicitly stated. As usual, we denote by $R$-Mod and Mod-$R$ the category of left $R$-modules and the category of right $R$-modules, respectively. We denote by $C(R)$ the category of complexes of left $R$-modules.

To any given class of objects $\mathcal{C}$ of $\mathscr{A}$ we associate its right Ext-orthogonal class, $$\mathcal{C}^\perp= \{X \in \mathscr{A} \mid \Ext^1(C,X) = 0, C\in \mathcal{C}\},$$ and its left Ext-orthogonal class, $${}^\perp\mathcal{C}= \{X \in \mathscr{A} \mid \Ext^1(X,C) = 0, C\in \mathcal{C}\}.$$ In particular, if $\mathcal{C}= \{M\}$ then we simply write ${}^\perp\mathcal{C}={}^\perp M$ and  $\mathcal{C}^\perp=M^\perp$.

Recall that, given a class of objects $\F$, an $\F$-precover of an object $M$ is a morphism $F\rightarrow M$ with $F\in \F$, such that $\Hom(F',F)\rightarrow \Hom (F',M)\rightarrow 0$ is exact for any $F'\in\F$. An $\F$-precover is said to be special provided that it is an epimorphism with kernel in the class $\FP$. $\F$-preenvelopes and special $\F$-preenvelopes are defined dually.

Recall that an object $M$ of an abelian category $\mathscr A$ is said to be Gorenstein projective if  there exists an exact and $\Hom(-, Proj(\A))$-exact complex of projective objects $$\cdots \to P_{-1} \to P_0 \to P_{1} \to \cdots$$ such that $M =\Ker(P_0 \to P_1)$ (see \cite[Definition 10.2.1]{Enochs}). While an object $M$ of an abelian category $\mathscr A$ is said to be strongly Gorenstein projective if there exists a short exact sequence $ 0 \to M \to P \to M \to 0$ with $P$ projective and $M \in {}^\perp{Proj}(\mathscr{A}) $ (see \cite{BM}). We use $\mathcal{GP}$ and $\mathcal{SGP}$ to denote the class of all Gorenstein projective objects and the class of all strongly Gorenstein projective objects, respectively.

In \cite{Sergio}, an alternative perspective on the projectivity of a module was introduced. Recall that, for two modules $M$ and $N$, $M$ is said to be $N$-subprojective if for every epimorphism $g : B \to N $ and every morphism $f : M \to  N$,  there exists a morphism $h : M \to B$  such that $gh = f$. Then, Holston et al. defined in \cite{Sergio} the subprojectivity domain of any module $M$ as the class of all modules $N$ such that $M$ is $N$-subprojective. The purpose of \cite{Sergio} was to introduce a new approach on the analysis of the projectivity of a module. However, the study of the subprojectivity goes beyond that aim and, indeed provides, among other things, an interesting new side on some other known notions. This opens a new important area of research which attracts many authors.

In this paper, we develop a new treatment  of  the subprojectivity in the categorical context. This study provides new interesting tools to develop this area of research. Indeed, we obtain, for instance, generalizations of several results using new methods which give a different light to the way they are seen now, which in addition, gives new perspectives. And we unify known and classical results in terms of subprojectivity. The current study provides also new powerful tools in constructing various interesting examples. For instance, we know and it is easy to show that the subprojectivity domain of a projective object $P$ is the whole category $\mathscr{A}$, which is exactly the  right Ext-orthogonal class of $P$. So we can write ${{\underline {\mathfrak{Pr}}}}^{-1}(P)=P^\perp$. So it is natural to ask how far we can get by extending this equality. We will show that, at least, it is possible to extend it to objects which are embedded in projective ones (see Proposition \ref{M'-Q-M}). As a consequence, we deduce that, if $M$ is a Gorenstein projective object, then there is an object $N$ such that  ${{\underline {\mathfrak{Pr}}}_{\mathscr{A}}}^{-1}(M)=N^\perp$.

We also introduce subprojectivity domains of classes as a natural extension of the subprojectivity domains of objects. This provides a new context in this domain of research in which several interesting questions arise. We show, among several other things, that the fact that a subprojectivity domain of a class coincides with its first right Ext-orthogonal class can be characterized in terms of preenvelopes and precovers.
  
The paper is organized as follows:

In section \ref{sec2}, we investigate subprojectivity domains of objects. We start by giving examples in the category of complexes and the category of representations of a quiver which show that the role of subprojectivity could go beyond the measure of the projectivity, but also have the ability to measure other properties such as the exactness of complexes (Proposition \ref{prop-semi-proj-ob}, see also Proposition \ref{prop-semi-proj-cl}) or when a morphism is monic (see Proposition \ref{prop-monic}). The main contribution of this section is the elaboration of two new ways to treat the subprojectivity of objects. The first one, Proposition \ref{Lemma-subproj}, is a functorial characterization of  the subprojectivity of objects and the second one, Proposition \ref{prop-facto}, characterizes the subprojectivity of objects in terms of factorizations of morphisms. This contribution allows to easily establish throughout the paper new and interesting results and examples. For instance, Corollary \ref{cor-subproj-sgoren-ob} shows that if $M$ is a strongly Gorenstein projective object then ${{\underline {\mathfrak{Pr}}}_{\mathscr{A}}}^{-1}(M)=M^\perp$. And Corollaries \ref{cor-emb-proj} and \ref{cor-fg-emb-proj} give, in terms of subprojectivity, a new way to see how an object (module) can be embedded in a projective object (module).

We also introduce and investigate subprojectivity domains of classes as a natural extension of subprojectivity domains of objects. This notion leads, among other things, to an unification of several well-known results (see Corollaries \ref{cor-A-flat}, \ref{cor-pp-A}, \ref{cor-A-fproj} and \ref{cor-A-simpleproj}). We determine subprojectivity domains of various classes such as the one of semi-projective complexes  (Proposition \ref{prop-semi-proj-cl}), the one of strongly Gorenstein projective objects (Proposition \ref{prop-subproj-sgoren-cl}), the one of finitely presented objects (Proposition \ref{prop-sub-fp-flat}), the one of finitely generated modules (Proposition \ref{prop-fg-fproj}), and the one of simple modules (Proposition \ref{prop-simple-proj}). We show in Proposition \ref{prop-sub-dom-Add} that the subprojectivity domain of a class $\mathcal{L}$ does not change even if we modify this class to  $Add(\mathcal{L})$ (i.e. the class of all objects which are isomorphic to direct summands of direct sums of copies of objects of the class $\mathcal{L}$). As consequences, the subprojectivity domains of the classes of all pure-projective modules, all  semisimple modules and all Gorenstein projective objects are determined (see Corollaries \ref{cor-pure-proj-flat}, \ref{prop-semisimple-proj} and \ref{prop-subproj-goren-cl}).

Section \ref{sec3} is devoted to the study of some closure properties of subprojectivity domains. We extend the study done in  \cite{Sergio} and we give new results. In Proposition \ref{extension} we show that the subprojectivity domain of any class is closed under extensions, finite direct sums and direct summands. Then, we characterize when are the subprojectivity domains closed under kernels of epimorphisms (Proposition \ref{lem-chara-cols-kernl} and Example \ref{ex-SGP-FG}). In Proposition \ref{clsub} we show that the subprojectivity domain of a class $\mathcal{L}$ is closed under subobjects if and only if the subprojectivity domain of any object of $\mathcal{L}$ is closed under subobjects. This leads to new characterizations of known notions. For instance, in Corollary \ref{cor-wdim-fp}, we show that, for any ring $R$,  $\wdim R \leq 1$ if and only if the subprojectivity domain of each finitely presented module is closed under submodules. In Corollary \ref{cor-R-semihered} we prove that a left coherent ring $R$ is left semihereditary if and only if the subprojectivity domain of each of its finitely generated modules is closed under submodules. Similarly, in Proposition \ref{clprod}, we generalize  \cite[Proposition 2.14]{Sergio} by showing that the subprojectivity domain of a class $\mathcal{L}$ is closed under arbitrary direct products if and only if the subprojectivity domain of any of its objects is closed under arbitrary direct products. This result allows us to give a much direct proof (see Corollary \ref{cor-R-coherent}) of a characterization of coherent rings established by  Durğun in \cite[Proposition 2.3]{Dur2}. Inspired by the work of Parra and Rada \cite{Rada-Parra}, we show that, if we assume further conditions on $\mathscr A$, then the closure under direct products of the subprojectivity domains of classes can be characterized in terms of preenvelopes (see Proposition \ref{prop-L-prec-proj}). We end Section \ref{sec3} with a discussion on the closure under direct sums of the subprojectivity domains of classes. In \cite[Proposition 2.13]{Sergio}, it was shown that the subprojectivity domain of any finitely generated module is closed under arbitrary direct sums. Here, using the functorial characterization of the subprojectivity domains, we show that this also holds for small objects.
  
Finally, in Section \ref{sec4} we relate subprojectivity domains with right Ext-orthogonal of classes. The main result (Theorem \ref{thm-main4}) states that, under some conditions on the category $\mathscr A$ and on the class $\mathcal{L}$, the following conditions are equivalent:
\begin{enumerate}
\item $ \mathcal{\mathcal{L}^\perp} ={{\underline {\mathfrak{Pr}}}_{\mathscr{A}}}^{-1}(\mathcal{L})$.
\item $ \mathcal{L} \bigcap \mathcal{\mathcal{L}^\perp} = Proj(\mathscr A)$ and every object in $\mathcal{L}^\perp$ has a special $\mathcal{L}$-precover.
\item    $Proj(\mathcal{A}) \subseteq \mathcal{L}^\perp$, ${{\underline {\mathfrak{Pr}}}_{\mathcal{A}}}^{-1}(\mathcal{L})$ is closed under cokernels of monomorphisms  and every $M \in \mathcal{L}$ has an $\mathcal{L}^\perp$-preenvelope which is projective.
\end{enumerate} 

We end the paper with two consequences (Corollaries \ref{cor-R-quasifrob} and \ref{cor-specialGP-precover}). In the last one we show that every object in $\mathcal{GP}^\perp$ has a special $\mathcal{GP}$-precover. This is in fact the  recent result \cite[Proposition 4.1]{ZH19} established in a different way.

 %
 %

\section{Subprojectivity domains in abelian categories} \label{sec2}

Subprojectivity of objects is a notion studied up to a certain level of deepness in categories of modules. However, it is a categorical type concept which has not even been considered in this general setting. The aim of this section is thus to explore the meaning of subprojectivity in nice categories from the homological point of view: abelian categories.

We start by recalling what subprojectivity means.

\begin{Def}[\cite{Sergio}]
Given two objects $M$ and $N$ in ${\mathscr{A}}$, $M$ is said to be $N$-subprojective if for every morphism $f:M\to N$ and every epimorphism $g:K \to N$, there exists a morphism $h:M\to K$ such that $gh=f$.

The \textit{subprojectivity domain}, or \textit{domain of subprojectivity}, of $M$ is defined as the class $${\underline {\mathfrak{Pr}}}_{\mathscr{A}}^{-1}(M):=\{ N \in {\mathscr{A}}: M \ is \ N\text {-}subprojective\}.$$
\end{Def}

In \cite[Lemma 2.3]{Sergio} it was proved that for $M$ to be $N$-subprojective one only needs to lift maps to projective modules that cover $N$ or even to a single projective module covering $N$. We now provide a functorial extension of this result which we will use to give some examples.

\begin{prop}\label{Lemma-subproj}
Let $M$ and $N$ be two objects of ${\mathscr{A}}$ and $\mathcal{X}$ be a subclass of ${\underline {\mathfrak{Pr}}}_{\mathscr{A}}^{-1}({M})$ such that every object in ${\underline {\mathfrak{Pr}}}_{\mathscr{A}}^{-1}(M)$ is an epic image of an object in $\mathcal{X}$. Then the following conditions are equivalent:
\begin{enumerate}
\item $M$ is $N$-subprojective. 
\item There exists a morphism $g:X\to N$ with $X \in \mathcal{X}$ such that $\Hom(M,g)$ is an epimorphism. 
\end{enumerate}
\end{prop}
\begin{proof} One can easily see that $(1)\Rightarrow (2)$.

To show $(2)\Rightarrow (1)$ assume that there exists such a morphism $g:X \to N$ and let $K \to N$ be an epimorphism. Then apply $\Hom(M,-)$ to the pullback diagram $$\xymatrix{D \ar[r] \ar[d] & X \ar[d] \ar[r] & 0 \\ K \ar[r] & N \ar[r] & 0}$$ to get $$\xymatrix{\Hom(M,D) \ar[r] \ar[d] & \Hom(M,X) \ar[d] \\ \Hom(M,K) \ar[r] & \Hom(M,N)}$$

By $(2)$ we see that $\Hom(M,X) \to \Hom(M,N)$ is an epimorphism and, since $\mathcal{X}$ is a subclass of ${\underline {\mathfrak{Pr}}}_{\mathscr{A}}^{-1}({M})$, the morphism $\Hom(M,D) \to \Hom(M,X)$ is also an epimorphism. Then $\Hom(M,K) \to \Hom(M,N)$ is of course epic.
\end{proof}

Notice that conditions $(1)$ and $(2)$ of Proposition \ref{Lemma-subproj}  are equivalent even if $\mathscr A$ does not have enough projectives, and that $\mathcal{X}$ could be the whole ${\underline {\mathfrak{Pr}}}_{\mathscr{A}}^{-1}({M})$ as considered in \cite[Proposition 2.1]{Dur2}, or the class of all  projective objects as in \cite[Lemma 2.3]{Sergio}.

We now give an example, in the context of representations of quivers by modules, of the usefulness of the above fact. Recall that the linear quiver $$v_n \to v_{n-1} \to \cdots \to v_2\to v_1$$ is denoted by $A_{n}$ and  the category of representations of $A_{n}$ is denoted by $Rep({A_n})$.   As in  \cite{EnoEs}, we use $\overline{M}[i]$, for a module $M$, to denote the representation 
$$0 \rightarrow 0  \rightarrow \cdots \rightarrow 0 \rightarrow M \stackrel{id}\longrightarrow  \cdots\stackrel{id}\longrightarrow  M  \stackrel{id}\longrightarrow  M$$
where the last $M$ is in the $i$'th place.\\
Following  \cite[Theorem 4.1]{EnoEs} (see also \cite[Section 2]{Sang}), we know that a representation $$ \quad    M_n\stackrel{f_{n-1}}\longrightarrow M_{n-1}\longrightarrow  \cdots \longrightarrow M_2\stackrel{f_1}\longrightarrow M_1$$ of $A_n$ is projective if and only if it is a direct sum of the following projective representations: 
 $$ \begin{array}{l} 
 \overline{P_1}[1]:\ 0\rightarrow0\rightarrow\dots\rightarrow0\rightarrow0\rightarrow P_1, \\ \\ \overline{P_2}[2]:\ 0\rightarrow0\rightarrow\dots\rightarrow0\rightarrow P_{2} \stackrel{id}\longrightarrow  P_{2}, \\ \\ \qquad  \qquad  \qquad \vdots \\ \\   \overline{P_n}[n] :\  P_n\stackrel{id}\longrightarrow P_n\stackrel{id}\longrightarrow  \cdots \stackrel{id}\longrightarrow  P_n\stackrel{id}\longrightarrow  P_n\stackrel{id}\longrightarrow  P_n,\end{array}$$
where the $P_i$'s are all projective modules. Thus, for a module $M$, 
the  representation  $\overline{M}[i]$  is projective if and only if $M$ is projective. We generalize this fact to the case of subprojectivity. 

\begin{prop}
			Let $M$ be a module and $(N,\delta)=N_n\stackrel{\delta_n}\rightarrow N_{n-1} \stackrel{\delta_{n-1}} \rightarrow\cdots\stackrel{\delta_3}\rightarrow N_2\stackrel{\delta_2}\rightarrow N_1$ be a representation of  $A_n$ ($n\geq 2$) in $R$-\Mod.  Then, for an integer $1\leq i\leq n$, $N  \in  {\underline {\mathfrak{Pr}}}_{Rep({A_n})}^{-1}(\overline{M}[i])$ if and only if $N_i \in {\underline {\mathfrak{Pr}}}_{R-\Mod}^{-1}(M)$. 		
\end{prop}

\begin{proof} For simplicity in notation we only prove the case of ${A}_2$ (${A}_n$ follows by the same arguments). Thus, we just need to discuss two cases: $0\to M$ and $M\stackrel{id}\longrightarrow M$.

1. Choose any representation ${\mathbf N}: N_2 \to N_1 \in  {\underline {\mathfrak{Pr}}}_{Rep(A_2)}^{-1}(0 \to M)$, an epimorphism $(\alpha,\beta):{\mathbf P}\to {\mathbf N}$ from a projective representation ${\mathbf P}:P_2 \to P_1 \in Rep(A_2)$, and any morphism of modules $f: M \to N_1$. Then, $(0,f):  \left(0\to M\right)  \to {\mathbf N} $ is a morphism of representations and the diagram $$
	
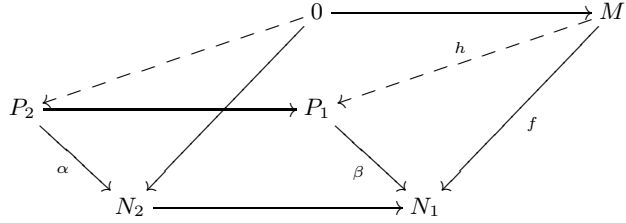
\xymatrix{
		&&& 0 \ar@{-->}[dlll]\ar[ddll]\ar[rrr]&&&M \ar@{-->}[dlll]_h \ar[ddll]^{f}\\
		P_2\ar[rd]_{\alpha}\ar[rrr]&&& P_1\ar[rd]_{\beta}&&&\\
		&N_2 \ar[rrr]&&&N_1&&\\
	}
	$$
can be completed commutatively. Therefore, $f=\beta h$ and, by Proposition \ref{Lemma-subproj}, $N_1 \in {\underline {\mathfrak{Pr}}}_{R-\Mod}^{-1}(M)$.

The converse is easy to prove.

2. To prove the necessary condition choose any representation $N_2\to N_1 \in  {\underline {\mathfrak{Pr}}}_{Rep({A}_2)}^{-1}(M \stackrel{id}\longrightarrow M)$ and any morphism of modules $f: M \to N_2$. Then, there exists a morphism of representations $(k,h)$ completing commutatively the diagram
	$$
	
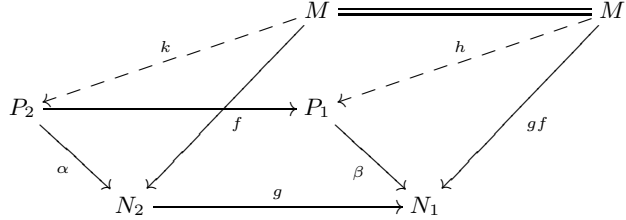
\xymatrix{
		&&& M \ar@{-->}[dlll]_k\ar[ddll]^f\ar@{=}[rrr] &&&M \ar@{-->}[dlll]_h \ar[ddll]^{gf}\\
		P_2\ar[rd]_{\alpha}\ar[rrr]&&& P_1\ar[rd]_{\beta}&&&\\
		&N_2 \ar[rrr]^g&&&N_1&&\\
	}
	$$ where $ P_2 \to P_1$  is a projective  representation and $  (\alpha, \beta)$ an epimorphism in the category $Rep({A}_2)$. Therefore $f=\alpha k$ and then, again by Proposition \ref{Lemma-subproj},  $N_2 \in {\underline {\mathfrak{Pr}}}_{R-\Mod}^{-1}(M)$.
	
Conversely, let $N_2\to N_1 $ be a representation in $Rep({A}_2)$. Suppose that $N_2$ is in ${\underline {\mathfrak{Pr}}}_{R-\Mod}^{-1}(M)$ and consider a projective representation $ P_2 \to P_1$, an epimorphism $(\alpha, \beta)$ from $P_2\to P_1$ onto $N_2\to N_1 $, and a morphism of representations $(f_2,f_1)$ from $M \stackrel{id}\longrightarrow M$ to $N_2\to N_1 $. Then,  there exists $h:M \to P_2$ such that $f_2=\alpha h$. Therefore, we get the following commutative diagram 
	$$
	
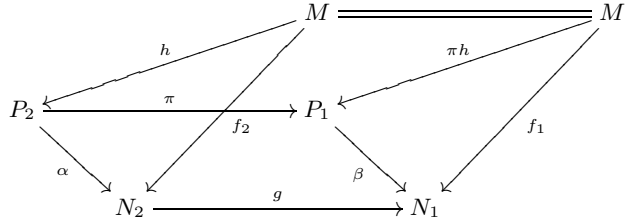
\xymatrix{
		&&& M \ar[dlll]_h\ar[ddll]^{f_2}\ar@{=}[rrr] &&&M \ar[dlll]_{\pi h} \ar[ddll]^{f_1}\\
		P_2\ar[rd]_{\alpha}\ar[rrr]^{\pi}&&& P_1\ar[rd]_{\beta}&&&\\
		&N_2\ar[rrr]^g&&&N_1&&\\
	}
	$$ so $f_1=g f_2= g \alpha h= \beta \pi h$ and hence $(f_2,f_1)=(\alpha, \beta) (h, \pi h)$. This means by Proposition \ref{Lemma-subproj} that $N_2 \to N_1 \in  {\underline {\mathfrak{Pr}}}_{Rep({A}_2)}^{-1}(M \stackrel{id}\longrightarrow M)$.
\end{proof}

Subprojectivity domains were introduced in \cite{Sergio} to, somehow, measure the projectivity of modules. So for instance it is clear that a module is projective precisely when its subprojectivity domain is the whole category $R$-Mod. Of course one immediately sees that this is not a situation which holds just in module categories. On the opposite, it does in every abelian category with enough projectives. We state it as a proposition.

\begin{prop}\label{prop-whole-cat-obj}
Let $M$ be an object of $\mathscr{A}$. Then the following conditions are equivalent:
\begin{enumerate}
\item  ${\underline {\mathfrak{Pr}}}_{\mathscr{A}}^{-1}(M)$ is the whole abelian category $\mathscr A$.
\item $M$ is projective.
\item $M\in{\underline {\mathfrak{Pr}}}_{\mathscr{A}}^{-1}(M)$.
\end{enumerate} 
\end{prop}

But in some cases subprojectivity can measure notions different from that of projectivity. We now give two examples showing this fact:

1.-In the category of complexes the subprojectivity domain of a semi-projective complex measures the exactness of such a complex (Proposition \ref{prop-semi-proj-ob}).

2.-Subprojectivity in the category of representations by modules of the quiver $A_2$ characterizes, in some cases, monomorphisms in $R$-Mod (Proposition \ref{prop-monic}).

Recall that a complex ${P}$ is said to be semi-projective (or DG-projective) if for every morphism $\alpha: P \to N$ and for every surjective quasi-isomorphism $\beta : M \to N$ there exists a morphism $\gamma : P \to M$ such that $\alpha = \beta\gamma$. Also, for every complex $N$ there exists a semi-projective complex $X$ and a quasi-isomorphism $f:X\to N$ (see \cite[Corollary 3.2.3]{Lars}).

\begin{prop} \label{prop-semi-proj-ob}
Let $N$ be a complex and let $f:X\to N$ be a quasi-isomorphism from a semi-projective complex $X$. Then, $N\in{{\underline {\mathfrak{Pr}}}_{C(R)}}^{-1}(X)$ if and only if $N$ is exact. 
\end{prop}
\begin{proof}
Suppose that $N \in {{\underline {\mathfrak{Pr}}}_{C(R)}}^{-1}(X)$ and consider an epimorphism $g:P\to N$ with $P$ projective. Since $N \in {{\underline {\mathfrak{Pr}}}_{C(R)}}^{-1}(X)$, there exists a morphism $h:X\to P$ such that $f=gh$, so $H(f)=H(g)H(h)$, and since $f$ is a quasi-isomorphism, $H(g): H(P)\to H(N)$ is an epimorphism, where $H(.)$ is the homology functor. But $P$ is exact, so $N$ is also exact.

Conversely, if $N$ is exact then $X$ is also exact since they are quasi-isomorphic, so by \cite[Proposition 2.3.7]{Ramon} $X$ is projective. Therefore, $N \in {{\underline {\mathfrak{Pr}}}_{C(R)}}^{-1}(X)$.
\end{proof}

\begin{prop}\label{prop-monic}
Let $M$ be a module and $g:N_1 \to N_2$ be a morphism of modules. Then, $$N_1 \stackrel{g}\rightarrow N_2 \in {\underline {\mathfrak{Pr}}}_{Rep(A_2)}^{-1}(M \to 0)\Leftrightarrow \Hom_{R}(M,\Ker g)=0.$$

In particular, if $M$ is free then $N_1 \stackrel{g}\rightarrow N_2 \in  {\underline {\mathfrak{Pr}}}_{Rep({A_2})}^{-1}(M \to 0)$ if and only if $g:N_1 \to N_2$ is monic.
\end{prop}

\begin{proof}
Suppose that $N_1 \stackrel{g}\rightarrow N_2 \in {\underline {\mathfrak{Pr}}}_{Rep({A_2})}^{-1}(M \to 0)$, let $f: M \to \Ker g$ be a morphism of modules and $i:\Ker g \to N_1$ be the canonical injection. We get the  following commutative diagram 
$$
\xymatrix{
&&& M \ar@{-->}[dlll]_h\ar[ddll]^{if}\ar[rrr]&&&0 \ar@{-->}[dlll] \ar[ddll]\\
P_1 \ar[rd]_{\beta}\ar[rrr]^{\pi}&&& P_2\ar[rd]^{\alpha}&&&\\
&N_1 \ar[rrr]_g&&&N_2&&\\
}
$$
where $P_1 \stackrel{\pi}\rightarrow P_2$ is a projective  representation and $(\beta, \alpha)$ is an epimorphism in the category $Rep({A}_2)$. Therefore, $\pi h=0$ so $f=0$ since $\pi$ is monic (see \cite[Theorem 3.1]{EnoEs}).

Conversely, if $\Hom_{R}(M,\Ker g)=0$ then $\Hom_{Rep({A_2})}(M\to 0,N_1\to N_2)=0$ so clearly $N_1 \stackrel{g}\rightarrow N_2\in  {\underline {\mathfrak{Pr}}}_{Rep({A_2})}^{-1}(M \to 0)$.
\end{proof}

The following result provides new ways to treat and use subprojectivity.

\begin{prop} \label{prop-facto}
Let $M$ and $N$ be objects of ${\mathscr{A}}$. Then,  $M$ is $N$-subprojective if and only if every morphism $M\to N$ factors through a projective object.
\end{prop}
\begin{proof} The necessary condition is clear since $\A$ has enough projectives.

Conversely, suppose that every morphism $M\to N$ factors through a projective object, consider any morphism $f:M\to N$ and any epimorphism $g:K\to N$. Then, by the assumption, there exists a projective object $P$ and a commutative diagram $$\xymatrix{P \ar[rd]_{\beta} & M \ar[l]_{\alpha} \ar[d]^f \\ K \ar[r]_g & N \ar[r] & 0}$$ And by the projectivity of $P$ the diagram $$\xymatrix{P \ar@{-->}[d]\ar[rd]_{\beta} & M \ar[l]_{\alpha} \ar[d]^f \\ K \ar[r]_g & N \ar[r] & 0}$$ can be completed commutatively.
\end{proof}

Last result leads to some interesting consequences. We start by the following.

\begin{prop}\label{M'-Q-M}
Let $0\to M \to Q\to M'\to 0$ be a short exact sequence with $Q$ projective. Then ${M'}^\perp\subseteq {{\underline {\mathfrak{Pr}}}_{\mathscr{A}}}^{-1}(M)$. If moreover $Proj(\mathscr{A})\subseteq {M'}^\perp$ then ${{\underline {\mathfrak{Pr}}}_{\mathscr{A}}}^{-1}(M)={M'}^\perp$.
\end{prop}

\begin{proof}
Let $N\in\mathscr{A}$ and consider the long exact sequence $$\xymatrix{\ar[r]& \Hom(Q,N) \ar[r]& \Hom(M,N) \ar[r]& \Ext^1(M',N)\ar[r]& \Ext^1(Q,N) \ar[r]&}$$

If  $N\in {M'}^\perp$ then $\Ext^1(M',N)=0$ so $\Hom(Q,N) \to \Hom(M,N)$ is epic, that is, any morphism $M\to N$ factors through the projective object $Q$. Then, by Proposition \ref{prop-facto} we deduce that $N\in {{\underline {\mathfrak{Pr}}}_{\mathscr{A}}}^{-1}(M)$. 

Suppose in addition that $Proj(\mathscr{A})\subseteq {M'}^\perp$. Let  $N\in {{\underline {\mathfrak{Pr}}}_{\mathscr{A}}}^{-1}(M)$ and let $P\to N$ be an epimorphism. We apply the functors $\Hom(-,P)$ and $\Hom(-,N)$ to $M\to Q$ to get the following commutative diagram $$
\xymatrix{
	\Hom(Q,P)\ar[r]\ar[d]& \Hom(M,P)\ar[d]\ar[r] & \Ext^1(M',P)\\
	\Hom(Q,N) \ar[r]& \Hom(M,N)\ar[r] & \Ext^1(M',N)\\
}
$$

To prove that $\Ext^1(M',N)=0$ it is sufficient to prove that $\Hom(Q,N)\to \Hom(M,N)$ is epic
since $Q$ is projective. But $N\in {{\underline {\mathfrak{Pr}}}_{\mathscr{A}}}^{-1}(M)$ so $\Hom(M,P) \to \Hom(M,N)$ is an epimorphism, and of course $\Hom(Q,P)\to\Hom(M,P)$ is epic ($Ext^1(M',P)=0$ by assumption), so then $\Hom(Q,N)\to \Hom(M,N)$ is epic.
\end{proof}

An example of an object satisfying the condition of Proposition \ref{M'-Q-M} can be found among strongly Gorenstein projective objects.

\begin{cor}\label{cor-subproj-sgoren-ob}  If $M$ is a  strongly Gorenstein projective object then  ${{\underline {\mathfrak{Pr}}}_{\mathscr{A}}}^{-1}(M)=M^\perp$.
\end{cor}

The converse of Corollary \ref{cor-subproj-sgoren-ob} does not hold in general. A counterexample can be found in commutative local artinian principal ideal rings: in \cite{BeHuWa15} it is proved that over one such a ring every module is $2$-strongly Gorenstein projective, that is, there exists an exact sequence $0 \to M \to P_2 \to P_1\to M \to 0$ with $P_1$ and $P_2$ projective and  $M \in {}^\perp{Proj}(\mathscr{A}) $. Thus, using \cite{BMO} one can prove that if the ring admits more than two proper ideals then the maximal ideal cannot be  strongly Gorenstein projective. So for instance, the ideal $(2+8\Z)$ of the ring $\Z/8\Z$ is not strongly Gorenstein projective. However we do have ${\underline {\mathfrak{Pr}}}_{(\Z/8\Z){\rm -Mod}}^{-1}(2+8\Z)=(2+8\Z)^\perp$ by the following result.

\begin{prop}
If $R$ is a commutative local  artinian principal ideal ring, then for every  module $M$, ${\underline {\mathfrak{Pr}}}_{R-\Mod}^{-1}(M)=M^\perp$.
\end{prop}  
\begin{proof} 
We can assume that $M$ is a non-projective module. 

By \cite[Proposition 4.5]{Sergio} we know that ${\underline {\mathfrak{Pr}}}_{R-\Mod}^{-1}(M)=Proj(R\mbox{-\Mod})$ so if we prove that $M^\perp=Proj(R\mbox{-\Mod})$ we will be done. 

Now, since $R$ is a commutative local artinian principal ideal ring, every module is a direct sum of cyclic modules, and the only composition series of the ring is $$0 = x^mR \subseteq  x^{m-1}R \subseteq  \cdots \subseteq  xR = Rad(R) \subseteq  R,$$ where $x$ is a generator of $Rad(R)$ (the Jacobson radical of $R$). Therefore, the result will follow if we show that $Proj(R\mbox{-\Mod})=(R/x^iR)^\perp$ for every $ 0<i<m$. 

Of course we have $Proj(R\mbox{-\Mod})\subseteq (R/x^iR)^\perp$ since $Proj(R\mbox{-\Mod})=Inj(R\mbox{-\Mod})$ ($R$ is a QF-ring). And on the other hand, if $N$ is a non projective module then there is an $i$ such that $R/x^iR$ is a direct summand of $R$. This means $\Ext_R(R/x^iR,R/x^jR)$ is a direct summand of $\Ext_R(N,R/x^jR)$ for every $j$. But $\Ext_R(R/x^iR,R/x^jR)\neq 0$ for all $0 < i,j < m$ by \cite[Example 4.5]{Tri} so we are done.
\end{proof}

Proposition \ref{M'-Q-M} says that if an object $M$ can be embedded in a projective object then there exists an object $M'$ such that $(M')^\perp\subseteq {\underline {\mathfrak{Pr}}}_{\mathscr{A}}^{-1}(M)$. Therefore, $ {\underline {\mathfrak{Pr}}}_{\mathscr{A}}^{-1}(M)$ contains the class of injective objects. This fact was proved using different arguments in \cite[Lemma 2.2]{Dur1} by giving a list of equivalences. The following result extends such a list of equivalent conditions.

\begin{cor}\label{cor-emb-proj} 
Assume that $\mathscr{A}$ has enough injectives and let $M$ be an object of $\mathscr A$. The following conditions are equivalent: 
\begin{enumerate}
	\item $M$ is   embedded in a projective object $P$.
	\item There exists an object $M'$ such that $(M')^\perp\subseteq {\underline {\mathfrak{Pr}}}_{\mathscr{A}}^{-1}(M)$.
	\item $Inj(\mathscr A)\subseteq  {\underline {\mathfrak{Pr}}}_{\mathscr{A}}^{-1}(M)$.
\end{enumerate}
\end{cor}
\begin{proof}
$(1)\Rightarrow(2)$ Proposition \ref{M'-Q-M}.

$(2)\Rightarrow(3)$ Clear.

$(3)\Rightarrow(1)$ Follows by similar arguments to those of \cite[Lemma 2.2]{Dur1}. 
\end{proof}

Now we prove that when the module is finitely generated (and can be embedded in a projective module) then its subprojectivity domain contains a larger class than that of the injectives. In what follows $FP$ will denote the class of all $FP$-injective modules, that is the class of right  Ext-orthogonal class of finitely presented modules.  Recall, from \cite[Proposition 6.2.4]{Enochs}, that the class $FP$ is
preenveloping. In particular,  every module embeds in an  $FP$-injective module.
 
\begin{cor}\label{cor-fg-emb-proj}
Let $M$ be a finitely generated module. The following conditions are equivalent:
\begin{enumerate}
	\item $M$ is   embedded in a projective module.
	\item There exists a finitely presented module $M'$ such that $(M')^\perp\subseteq  {{\underline {\mathfrak{Pr}}}}_{R-\Mod}^{-1}(M)$.
	\item $FP\subseteq  {\underline {\mathfrak{Pr}}}_{R-\Mod}^{-1}(M)$.
		\item For an $FP$-injective preenvelope $i: M \hookrightarrow  E$ of $M$,  $E \in  {\underline {\mathfrak{Pr}}}_{R-\Mod}^{-1}(M)$.
\end{enumerate}
\end{cor} 
\begin{proof}
$(1)\Rightarrow(2)$. Since $M$ is finitely generated, $M$ can be embedded in a finitely generated projective  module $F$, so $F/M$ is finitely presented. 

$(2)\Rightarrow(3)$ holds  since  $FP\subseteq  (M')^\perp $.

$(3)\Rightarrow(4)$ Clear.

$(4)\Rightarrow(1)$. Let $g: P \to E$ be an epimorphism with $P$ projective. Since $E\in  {\underline {\mathfrak{Pr}}}_{R-\Mod}^{-1}(M)$, the diagram $$\xymatrix{ & M \ar[d]^i \ar@{-->}[ld]_h \\ P \ar[r]_g &  E \ar[r] & 0}$$ can be completed commutatively by $h$. Thus, $h$ must be injective.
\end{proof}
 
 The class of pure-injective modules is another interesting class which also contains the class of injective modules. Thus, it is natural to ask whether there is the ``pure-injective" version of Corollary \ref{cor-fg-emb-proj}.  This question seems not to have a direct proof as the one given for Corollary \ref{cor-fg-emb-proj}	since, unlike the $FP$-injective modules which can be defined as the right  Ext-orthogonal class of finitely presented modules, the pure-injective modules are defined, in the context of relative homological algebra, as  the injective objects with respect to pure-exact sequences. This requires to add, as pointed out by the referee, the purity condition on the morphisms.\\
 Let us denote by $PE$ the class of all  pure-injective modules.  
 
 \begin{prop}\label{pro-emb-proj}
Let $M$ be a  module. The following conditions are equivalent:
\begin{enumerate}
	\item $M$ is purely embedded in a projective module; that is, there is a pure monomorphism $M  \hookrightarrow  P$ for some projective module $P$. 
	\item $PE\subseteq  {\underline {\mathfrak{Pr}}}_{R-\Mod}^{-1}(M)$.
		\item For any $PE$-injective preenvelope $i: M \hookrightarrow  E$ of $M$,  $E \in  {\underline {\mathfrak{Pr}}}_{R-\Mod}^{-1}(M)$.
\end{enumerate}
\end{prop} 
\begin{proof}
 $(1)\Rightarrow(2)$. Let $E$ be a pure-injective module and $f:M\to E$ be any morphism. By hypothesis, there is a pure monomorphism $\alpha:M  \hookrightarrow  P$ for some projective module $P$. Then, there exists a morphism $\beta: P \to  E$ such that $ \beta\alpha =f $, which shows that $f$ factors through a projective module and so $  E\in {\underline {\mathfrak{Pr}}}_{R-\Mod}^{-1}(M)$.
 
 $(2)\Rightarrow(3)$ Clear.
 
 $(3)\Rightarrow(1)$. Let $PE(M)$ denote the pure-injective envelope of $M$. Then, there is a pure monomorphism $\iota :M  \hookrightarrow  PE(M)$.
 Let $g:P\to PE(M)$ be an epimorphism with $P$ is projective. Since $PE(M)\in  {\underline {\mathfrak{Pr}}}_{R-\Mod}^{-1}(M)$,  there exists a morphism $h: M\to P$ such that $gh=\iota$; in particular, $h$ is a pure monomorphism, as desired. 
 \end{proof}

 Notice that this last fact might also hold in more general contexts than those of modules by considering Herzog's main result of \cite{Herzog}, which shows that if $\mathscr{A}$ is a locally finitely presented additive category, then every object of $\mathscr{A}$ admits a pure-injective envelope. \\

We now fix our attention on classes of objects: we introduce and investigate subprojectivity domains of classes instead of just single objects. The subprojectivity domain of a class $X$ is defined as the class of all objects holding in the subprojectivity domain of all objects of $X$.

\begin{Def}
The \textit{subprojectivity domain}, or \textit{domain of subprojectivity}, of a class of objects $\mathcal{M}$ of $\A$ is defined as $${\underline {\mathfrak{Pr}}}_{\mathscr{A}}^{-1}(\mathcal{M}):=\{ N \in {\mathscr{A}}: M \ is \ N \text {-subprojective for every  } M \in \mathcal{M}\}.$$
\end{Def}

Therefore, if $\mathcal{M} := \{M\}$ then  ${{\underline {\mathfrak{Pr}}}_{\mathscr{A}}}^{-1}(\mathcal{M})={{\underline {\mathfrak{Pr}}}_{\mathscr{A}}}^{-1}(M).$

Proposition \ref{prop-whole-cat-obj} characterizes when the subprojectivity domain of an object is the whole abelian category $\mathscr{A}$. The following extension to classes of such a proposition can be used to unify various classical results.

\begin{prop}\label{prop-whole-cat}
Let $\mathcal{L}$ be a class of objects of $\mathscr{A}$. Then the following conditions are equivalent:
\begin{enumerate}
\item  ${\underline {\mathfrak{Pr}}}_{\mathscr{A}}^{-1}(\mathcal{L})$ is the whole abelian category $\mathscr A$.
\item Every object of $\mathcal{L}$ is projective.
\item  $\mathcal{L} \subseteq {\underline {\mathfrak{Pr}}}_{\mathscr{A}}^{-1}(\mathcal{L})$.
\end{enumerate} 
\end{prop}
\begin{proof} To prove $(1) \Rightarrow (2) $, let $L$ in $\mathcal{L}$ and $P \to L$ be an epimorphism with $P$ is projective. Since  $L \in {\underline {\mathfrak{Pr}}}_{\mathscr{A}}^{-1}(L)$, $P \to L$ splits. Hence $L$ is projective.

The implication $(2) \Rightarrow (3) $ is clear since ${\underline {\mathfrak{Pr}}}_{\mathscr{A}}^{-1}(\mathcal{L})$ contains the class of projectives.

To prove $(3) \Rightarrow (1) $, consider an object $L$ in $\mathcal{L}$. By assumption  $L \in {\underline {\mathfrak{Pr}}}_{\mathscr{A}}^{-1}(L)$, hence $L$ is projective (Proposition \ref{prop-whole-cat-obj}). So ${\underline {\mathfrak{Pr}}}_{\mathscr{A}}^{-1}(L)$ coincide with $\mathscr A$ for any $L$ in $\mathcal{L}$. Therefore, ${\underline {\mathfrak{Pr}}}_{\mathscr{A}}^{-1}(\mathcal{L})$ is $\mathscr A$.
\end{proof}

It is clear that if $X$ is a subclass of a class $Y$, then ${\underline {\mathfrak{Pr}}}_{\mathscr{A}}^{-1}(Y) \subseteq  {\underline {\mathfrak{Pr}}}_{\mathscr{A}}^{-1}(X)$.  The following results show how far we can modify classes while preserving the same  subprojectivity domain.

We start with a result which shows that reducing classes to a singleton while preserving the same subprojectivity domain is possible. This is based on the following observation. Recall that if $S$ is the representative set of finitely presented modules, then it is known that a module $F$ is flat if and only if $\Hom(\oplus_{M\in S} M,-)$ makes exact every short exact sequence of the form $0 \to A \to B \to F \to 0$. This means that ${{\underline {\mathfrak{Pr}}}_{\mathscr{A}}}^{-1}(\oplus_{M\in S} M)$ is the class of flat modules (as proved in \cite[Proposition 2.1]{Dur1}). The following result, which was already proven for the category of modules in \cite[Proposition 2.10]{Sergio}, is a generalization of this fact.

\begin{prop} \label{dsum}
Suppose that $\mathscr{A}$ has direct sums and let $\{M_i\}_{i\in I}$ be a set of objects in $\mathscr{A}$. Then ${{\underline {\mathfrak{Pr}}}_{\mathscr{A}}}^{-1}(\oplus_{i\in I}M_i)={{\underline {\mathfrak{Pr}}}_{\mathscr{A}}}^{-1}(\{M_i\}_{i\in I})$.
\end{prop}
\begin{proof}
Let $g: K\to N$ be an epimorphism. The following diagram is commutative $$\xymatrix{\Hom(\oplus_{i\in I} M_i, K)  \ar[d]_{\psi^K}  \ar[rr]^{\Hom(\oplus_{i\in I} M_i, g)} && \Hom(\oplus_{i\in I} M_i, N) \ar[d]^{\psi^N} \\ \prod_{i\in I}\Hom( M_i, K) \ar[rr] & &\prod_{i\in I}\Hom( M_i, N)}$$ where $\psi^K$ and $\psi^N$ are isomorphisms. Hence the morphism ${\Hom(\oplus_{i\in I} M_i, g)}$ is epic if and only if $\prod_{i\in I}\Hom( M_i,g)$ is epic. Therefore $N\in {{\underline {\mathfrak{Pr}}}_{\mathscr{A}}}^{-1}(\oplus_{i\in I}M_i)$ if and only if $N\in {{\underline {\mathfrak{Pr}}}_{\mathscr{A}}}^{-1}(M_i)$ for every ${i\in I}$.
\end{proof}

We now give an extension of Proposition \ref{dsum}. For we will use the following known terminology: if $\mathcal{L}$ is a class of objects of $\mathscr A$, we denote by $Sum (\mathcal{L})$ the class of all objects which are isomorphic to direct sums of objects of $\mathcal{L}$, by $Summ (\mathcal{L})$ the class of all objects which are isomorphic to direct summands of objects of $\mathcal{L}$, and by $Add(\mathcal{L})$ the class $Summ (Sum (\mathcal{L}))$.

\begin{prop}\label{prop-sub-dom-Add}
Let $\mathcal{L}$ be a class of objects of $\mathscr A$. Then $${{\underline {\mathfrak{Pr}}}_{\mathscr{A}}}^{-1}(Add(\mathcal{L}))= {{\underline {\mathfrak{Pr}}}_{\mathscr{A}}}^{-1}(Sum(\mathcal{L}))={{\underline {\mathfrak{Pr}}}_{\mathscr{A}}}^{-1}(Summ(\mathcal{L}))={{\underline {\mathfrak{Pr}}}_{\mathscr{A}}}^{-1}(\mathcal{L}).$$

If $\mathcal{L}$ is a set, then all these classes coincide with the class $ {{\underline {\mathfrak{Pr}}}_{\mathscr{A}}}^{-1}(\oplus_{L\in \mathcal{L}}L).$
\end{prop}
\begin{proof}
It is clear that ${{\underline {\mathfrak{Pr}}}_{\mathscr{A}}}^{-1}(Add(\mathcal{L}))$ holds inside ${{\underline {\mathfrak{Pr}}}_{\mathscr{A}}}^{-1}(\mathcal{L})$ since $\mathcal{L}$ holds inside $Add(\mathcal{L})$. 

Conversely, let $N$ be in ${{\underline {\mathfrak{Pr}}}_{\mathscr{A}}}^{-1}(\mathcal{L})$ and $M$ in $Add (\mathcal{L})$. Then, there exist $M'$ in $Add (\mathcal{L})$ and a family $\{L_i\}$ in $\mathcal{L}$ such that $M \oplus M' = \oplus_i L_i$.  By Proposition \ref{dsum}, $N \in {{\underline {\mathfrak{Pr}}}_{\mathscr{A}}}^{-1}(M)$ so  $N \in {{\underline {\mathfrak{Pr}}}_{\mathscr{A}}}^{-1}(Add(\mathcal{L}))$. Therefore, ${{\underline {\mathfrak{Pr}}}_{\mathscr{A}}}^{-1}(Add(\mathcal{L}))= {{\underline {\mathfrak{Pr}}}_{\mathscr{A}}}^{-1}(\mathcal{L})$. 

Now, it is clear that $\mathcal{L} \subseteq  Summ (\mathcal{L}) \subseteq  Add(\mathcal{L})$ and that $\mathcal{L} \subseteq  Sum (\mathcal{L}) \subseteq  Add(\mathcal{L})$, so we have ${{\underline {\mathfrak{Pr}}}_{\mathscr{A}}}^{-1}(Add(\mathcal{L})) \subseteq   {{\underline {\mathfrak{Pr}}}_{\mathscr{A}}}^{-1}(Summ(\mathcal{L})) \subseteq  {{\underline {\mathfrak{Pr}}}_{\mathscr{A}}}^{-1}(\mathcal{L})$ and ${{\underline {\mathfrak{Pr}}}_{\mathscr{A}}}^{-1}(Add(\mathcal{L})) \subseteq   {{\underline {\mathfrak{Pr}}}_{\mathscr{A}}}^{-1}(Sum(\mathcal{L})) \subseteq  {{\underline {\mathfrak{Pr}}}_{\mathscr{A}}}^{-1}(\mathcal{L})$, so we conclude that ${{\underline {\mathfrak{Pr}}}_{\mathscr{A}}}^{-1}(Add(\mathcal{L}))= {{\underline {\mathfrak{Pr}}}_{\mathscr{A}}}^{-1}(Sum(\mathcal{L}))={{\underline {\mathfrak{Pr}}}_{\mathscr{A}}}^{-1}(Summ(\mathcal{L}))={{\underline {\mathfrak{Pr}}}_{\mathscr{A}}}^{-1}(\mathcal{L})$.

If $\mathcal{L}$ is a set then, by Proposition \ref{dsum}, ${{\underline {\mathfrak{Pr}}}_{\mathscr{A}}}^{-1}(\mathcal{L})= {{\underline {\mathfrak{Pr}}}_{\mathscr{A}}}^{-1}(\oplus_{L\in \mathcal{L}}L)$.
\end{proof}

As mentioned before, Proposition \ref{prop-whole-cat} can be used to unify various known results. These will follow from establishing the subprojectivity domains of the following five well known classes of modules.

The case of the class of finitely presented objects can be deduced directly from the categorical definition of flat objects.  Indeed, an object $F$ is said to be flat if every short exact sequence $0 \to A \to B \to F \to 0$ is pure, that is, if for every finitely presented object $P$, $\Hom_{\A}(P, -)$ makes this sequence exact (see \cite{Sten}). Then, we get the following result (see also \cite[Proposition 2.1]{Dur1}).  

\begin{prop} [\cite{Dur1}, Proposition 2.1] \label{prop-sub-fp-flat}
	The subprojectivity domain of the class of finitely presented objects is the class of flat objects.
\end{prop}

Now, applying Proposition \ref{prop-whole-cat} to the class of finitely presented objects we get the following well known result.

\begin{cor}\label{cor-A-flat}
The following conditions are equivalent:
\begin{enumerate}
	\item Every object of $\mathscr{A}$ is flat.
	\item Every finitely presented object is projective.
	\item Every finitely presented object is flat.
\end{enumerate}
\end{cor}

Recall now that an object is said to be pure-projective if it is projective with respect to every pure short exact sequence. If the category is locally finitely presented then using the same arguments in (\cite[Corollary 3]{Warfield}) one can show that an object is pure-projective if and only if it is a direct summand of a direct sum of finitely presented objects. As a direct consequence of Proposition \ref{prop-sub-dom-Add}, we get the following result.

\begin{cor} \label{cor-pure-proj-flat}
If the category is locally finitely presented then the subprojectivity domain of the class of all pure-projective objects is precisely the class of all flat objects.
\end{cor}

Then, by Proposition \ref{prop-whole-cat} applied to the class of pure-projective objects we get the following known result (see \cite{Fieldhouse}). 

\begin{cor}\label{cor-pp-A}
If the category $\A$ is locally finitely presented, then the following conditions are equivalent:
\begin{enumerate}
\item Every object of $\A$ is flat.
\item Every pure-projective object is projective.
\item Every pure-projective object is flat.
\end{enumerate}
\end{cor}

Proposition \ref{prop-facto} relates belonging to some subprojectivity domain with factorization of morphisms through projective objects, and this suggests studying subprojectivity domains of classes defined by means of factorizations. 

Recall that a module $M$ is said to be f-projective if for every finitely generated submodule $C$ of $M$, the inclusion map $C\to M$ factors through a finitely generated free module. Then, we have the following result.

\begin{prop} \label{prop-fg-fproj}
The subprojectivity domain of the class of finitely generated modules is the class of f-projective modules.
\end{prop}
\begin{proof}
Let $N$ be an $f$-projective module and $M$ be a finitely generated module. The image of any morphism $f:M\to N$, $\Im(f)$, is a finitely generated submodule of $N$, so being $N$ f-projective means that the inclusion map $\Im(f)\to N$ factors through a projective module and so that $M$ is $N$-subprojective (Proposition \ref{prop-facto}). 

Conversely, let $N$ be in the subprojectivity domain of the class of finitely generated modules, $N'$ be a finitely generated submodule of $N$ and $i: N'\to N$ the inclusion map. Let $g:F\to N$ be an epimorphism  with $F$ free. Since $N \in {\underline {\mathfrak{Pr}}}_{R-\Mod}^{-1}({N'})$, there exists a morphism $h:N'\to F$ such that $i=gh$. Since $\Im( h)$ is finitely generated, there exists a finitely generated free module $F'$ such that $\Im (h)\subseteq  F' \subseteq  F$. Hence $i$ factors through $F'$ and $N$ is f-projective.
\end{proof}

Now, applying Proposition \ref{prop-whole-cat} to the class of finitely generated modules we get the following.

\begin{cor} \label{cor-A-fproj}
The following conditions are equivalent:
\begin{enumerate}
\item Every module is f-projective.
\item Every finitely generated module is projective, that is $R$ is a semisimple
artinian ring.  
\item Every finitely generated module is f-projective.
\end{enumerate}
\end{cor} 

In a similar way to Proposition \ref{prop-fg-fproj}, we can determine the subprojectivity domain of simple modules. Recall that a module $N$ is called simple-projective if, for any simple module $M$, every morphism $f:M\to N$ factors through a finitely generated free module (see \cite[Definition 2.1]{Mao}).

\begin{prop} \label{prop-simple-proj}
The subprojectivity domain of the class of simple modules is the class of simple-projective modules. 
\end{prop}

And again, using Proposition \ref{prop-sub-dom-Add}, we have the subprojectivity domain of the class of all semisimple modules.

\begin{cor} \label{prop-semisimple-proj}
The subprojectivity domain of the class of semisimple modules is also the class of simple-projective modules.
\end{cor}

Then we get the following equivalences.

\begin{cor} \label{cor-A-simpleproj}
The following conditions are equivalent:
\begin{enumerate}
\item Every  module is simple-projective.
\item Every simple module is projective.
\item Every simple module is simple-projective.
\item Every semisimple module is projective.
\item Every semisimple module is simple-projective.
\end{enumerate}
\end{cor}

Let us finish this section by giving three more examples of subprojectivity domains of homologically important classes of objects.

We start by noticing that Proposition \ref{prop-semi-proj-ob} helps us deduce that ${{\underline {\mathfrak{Pr}}}_{C(R)}}^{-1}(\mathcal{SP})$ is a subclass of the class of exact complexes, where $\mathcal{SP}$ denotes the class of semi-projective complexes. The next proposition shows that we have an equality. 

\begin{prop} \label{prop-semi-proj-cl}
	The subprojectivity domain of the class of semi-projective complexes is the class of all exact complexes.
\end{prop}
\begin{proof}
	Let $E$ be an exact complex. By \cite[Corollary 3.2.3]{Lars}, there exists a surjective quasi-iso\-mor\-phism $g:P\to E$ where $P$ is semi-projective. Since $E$ is exact then $P$ is so, hence $P$ is projective. Thus, for every $M \in \mathcal{SP}$,  $\Hom(M,g)$ is epic. Therefore, $E \in {{\underline {\mathfrak{Pr}}}_{C(R)}}^{-1}(\mathcal{SP})$ by Proposition \ref{Lemma-subproj}.
\end{proof}

The case of the class of strongly Gorenstein projective objects can be deduced directly from Corollary \ref{cor-subproj-sgoren-ob}.

\begin{prop} \label{prop-subproj-sgoren-cl}
The subprojectivity domain of the class of strongly Gorenstein projective objects is the class  $\mathcal{SGP}^\perp$. 
\end{prop}

Finally, using Proposition \ref{prop-sub-dom-Add} we determine the subprojectivity domain of the class of Gorenstein projective objects. If direct sums exist and they are exact then, using similar arguments to those of \cite{BM}, we can show that an object is Gorenstein projective if and only if it is a direct summand of a strongly Gorenstein projective one, so clearly $\mathcal{SGP}^\perp = \mathcal{GP}^\perp$. Thus, we have the following result.

\begin{cor} \label{prop-subproj-goren-cl}
	If direct sums exist and they are exact then, the subprojectivity domain of the class of Gorenstein projective objects is the class $\mathcal{GP}^\perp$.
In particular, if $R$ is a ring  with  finite Gorenstein global dimension \cite{BM2},  then the subprojectivity domain of the class of Gorenstein projective modules  is the class of all modules with finite projective dimension. 
\end{cor}
\begin{proof}
The second assertion follows from the known fact that over a ring $R$ with finite Gorenstein
global dimension,  the class $\mathcal{GP}^\perp$   coincide with the class of all modules with finite projective dimension.  We do not have a precise reference but one can see that it is a simple consequence of  \cite[Lemma 2.17]{Lars2}.
\end{proof}


\section{Closure properties of the subprojectivity domains} \label{sec3}
The aim of this section is to investigate the closure properties of subprojectivity domains. This study leads to some new characterizations of known notions.

We start with the following generalization of \cite[Proposition 3]{Alagoz}, \cite[Proposition 2.11]{Sergio} and \cite[Proposition 2.12]{Sergio}. Though it can be proved by using similar arguments to those of the results it generalizes, we give an alternative proof since we think it provides new and useful ideas.

\begin{prop} \label{extension}
The subprojectivity domain of any class in  $\mathscr{A}$  is closed  under extensions, finite direct sums and direct summands.
\end{prop}
\begin{proof} Clearly it suffices to prove the result for subprojectivity domain of objects so let us consider a single object $M$ of $\A$ and study its subprojectivity domain.

For let $0 \rightarrow A \rightarrow B \rightarrow C \rightarrow 0$ be a short exact sequence of objects and suppose that  $A$ and $C$ are in ${{\underline {\mathfrak{Pr}}}_{\mathscr{A}}}^{-1}(M)$. Consider then two epimorphisms $P_A \to A$ and $P_C \to C$ with $P_A$ and $P_C$ projective. By Horseshoe Lemma we get the following commutative diagram $$\xymatrix{0\ar[r]& P_A \ar[r]\ar[d]&P_B\ar[r]\ar[d]&P_C\ar[r]\ar[d]&0\\
0\ar[r]& A \ar[r]&B\ar[r]&C\ar[r]&0 \\}$$ with $P_B$ is projective. Apply then $\Hom_{\A} (M,-)$ to get the commutative diagram $$\xymatrix{0\ar[r]& \Hom (M,P_A) \ar[r]\ar[d]&\Hom (M,P_B)\ar[r]\ar[d]&\Hom (M,P_C)\ar[r]\ar[d]&0 \\ 0\ar[r]& \Hom (M,A) \ar[r]&\Hom (M,B)\ar[r]&\Hom (M,C)\ar[r]&0}$$ with exact rows ($C$ holds in ${{\underline {\mathfrak{Pr}}}_{\mathscr{A}}}^{-1}(M)$).

Since $A$ and $C$ hold in ${{\underline {\mathfrak{Pr}}}_{\mathscr{A}}}^{-1}(M)$, the two morphisms $\Hom (M,P_A) \rightarrow \Hom (M,A)$ and $\Hom (M,P_C) \rightarrow \Hom (M,C)$ are epimorphisms. Hence, $\Hom (M,P_B) \rightarrow \Hom (M,B)$ is also an epimorphism and then we get $B \in {{\underline {\mathfrak{Pr}}}_{\mathscr{A}}}^{-1}(M)$ (by  Proposition \ref{Lemma-subproj}). 

Now, the closure under extensions of ${{\underline {\mathfrak{Pr}}}_{\mathscr{A}}}^{-1}(M)$ proves its closure under finite direct sums.

And finally, let $N\in {{\underline {\mathfrak{Pr}}}_{\mathscr{A}}}^{-1}(M)$ and $A$ be a direct summand of $N$. If $p: N \to A$ is the canonical projection then $\Hom(M, p)$ is epic and then, by   Proposition \ref{Lemma-subproj}, we get that $A \in {{\underline {\mathfrak{Pr}}}_{\mathscr{A}}}^{-1}(M)$.
\end{proof}

For subprojectivity domains that are closed  under kernels of epimorphisms, we have the following result.

\begin{prop} \label{lem-chara-cols-kernl}
Let $\mathcal{L}$ be a class of objects of $\mathscr{A}$. Then the following conditions are equivalent: 
\begin{enumerate}
\item ${\underline {\mathfrak{Pr}}}_\mathcal{A}^{-1}(\mathcal{L})$ is closed under kernels of epimorphisms.
\item For every short exact sequence $0\to C\to P \to A \to 0$ where $P$ is projective, if $A \in{\underline {\mathfrak{Pr}}}_\mathcal{A}^{-1}(\mathcal{L})$ then $C \in{\underline {\mathfrak{Pr}}}_\mathcal{A}^{-1}(\mathcal{L})$.
\item   For every epimorphism $P\to A$ with $P$ is projective and $A \in{\underline {\mathfrak{Pr}}}_\mathcal{A}^{-1}(\mathcal{L})$, the pullback object of $P$ over $A$ holds in ${\underline {\mathfrak{Pr}}}_\mathcal{A}^{-1}(\mathcal{L})$.
\end{enumerate}
\end{prop}
\begin{proof}
$ (1)\Rightarrow (2)$ is clear. To prove $ (2)\Rightarrow (1)$ consider an exact sequence $$0\to C\to B\to A\to 0$$ with $B,\ A\in {\underline {\mathfrak{Pr}}}_\mathcal{A}^{-1}(\mathcal{L})$ and the pullback diagram $$
\xymatrix{
&&  0 \ar[d]&  0 \ar[d]&\\
	 &&  K \ar@{=}[r] \ar[d]&  K \ar[d]\\
	0\ar[r]& C \ar[r] \ar@{=}[d]&  D \ar[r] \ar[d]&  P \ar[r] \ar[d]&  0\\
	0\ar[r]& C \ar[r] &  B \ar[r] \ar[d] &  A \ar[r] \ar[d] &  0\\
	&&0&0&
}
$$ where $P$ is a projective object. $A \in{\underline {\mathfrak{Pr}}}_\mathcal{A}^{-1}(\mathcal{L})$, so by assumption $K \in{\underline {\mathfrak{Pr}}}_\mathcal{A}^{-1}(\mathcal{L})$. Then, by Proposition \ref{extension}, $D \in{\underline {\mathfrak{Pr}}}_\mathcal{A}^{-1}(\mathcal{L})$, and since $C$ is a direct summand of $D$, we deduce using again Proposition \ref{extension} that $C \in{\underline {\mathfrak{Pr}}}_\mathcal{A}^{-1}(\mathcal{L})$.

To prove $ (3)\Leftrightarrow (2)$, consider the following diagram where $D$ is the pullback of $P$ over $A$ $$\xymatrix{
0\ar[r]& C \ar[r] \ar@{=}[d]&  D \ar[r] \ar[d]&  P \ar[r] \ar[d]&  0\\
	0\ar[r]& C \ar[r] &  P \ar[r]&  A \ar[r] &  0
}$$

Suppose that $A \in{\underline {\mathfrak{Pr}}}_\mathcal{A}^{-1}(\mathcal{L})$. By Proposition \ref{extension}, we have $D \in{\underline {\mathfrak{Pr}}}_\mathcal{A}^{-1}(\mathcal{L})$ if and only if $C \in{\underline {\mathfrak{Pr}}}_\mathcal{A}^{-1}(\mathcal{L})$.
\end{proof}

As examples of classes satisfying the conditions of Proposition \ref{lem-chara-cols-kernl}, we give the following.

\begin{ex}\label{ex-SGP-FG}
\begin{enumerate}
\item Let $M$ be a strongly Gorenstein projective object. Then ${\underline {\mathfrak{Pr}}}_\mathcal{A}^{-1}(M)$ is closed under kernels of epimorphisms.

Indeed, let $0\to C\to P \to A \to 0 $ be a short exact sequence with $P$ projective and $A\in{\underline {\mathfrak{Pr}}}_\mathcal{A}^{-1}(M)$. If we consider the long exact sequence $$\cdots \to \Hom(M,P) \to \Hom(M,A) \to \Ext^1(M,C) \to \Ext^1(M,P) \to \cdots,$$ then $\Ext^1(M,C)=0$ so $C\in{\underline {\mathfrak{Pr}}}_\mathcal{A}^{-1}(M)$ (see Corollary \ref{cor-subproj-sgoren-ob}). Therefore, by Proposition \ref{lem-chara-cols-kernl}, ${\underline {\mathfrak{Pr}}}_\mathcal{A}^{-1}(M)$ is closed under kernels of epimorphisms.

\item Let $\mathcal{L}$ be any class of finitely generated modules containing all finitely presented modules. Then, the subprojectivity domain of $\mathcal{L}$ is closed under kernels of epimorphisms. In particular, the class of f-projective modules is closed under kernels of epimorphisms.

To show this, let $0 \to A \to B \to C \to 0$ be a short exact sequence with $B,\ C \in {{\underline {\mathfrak{Pr}}}_{R-\Mod}}^{-1}(\mathcal{L})$. Since $\mathcal{L}$ contains all finitely presented modules the sequence $0 \to A \to B \to C \to 0$ is pure, so by \cite[Proposition 2.6]{Rada-Parra} $ A \in {{\underline {\mathfrak{Pr}}}_{R-\Mod}}^{-1}(\mathcal{L})$. 
\end{enumerate}
\end{ex}

In \cite[Proposition 2.15]{Sergio} it is proved that a ring $R$ is right hereditary if and only if the subprojectivity domain of any right $R$-module is closed under submodules. Since ${\underline {\mathfrak{Pr}}}_{\Mod-R}^{-1}(\mbox{\Mod-}R)$ is the class of projective right $R$-modules, one could replace the statement  ``$R$ is right hereditary" by ``${\underline {\mathfrak{Pr}}}_{\Mod-R}^{-1}(\mbox{\Mod-}R)$ is closed under submodules", getting then that ${\underline {\mathfrak{Pr}}}_{Mod-R}^{-1}(\mbox{\Mod-}R)$ is closed under submodules if and only if ${\underline {\mathfrak{Pr}}}_{Mod-R}^{-1}(M)$ is closed under submodules for every right $R$-module $M$. Thus,  the next proposition gives an extension  of this result to an arbitrary class $\mathcal{L}$ of objects of $\mathscr A$. 

\begin{prop} \label{clsub} 
Let $\mathcal{L}$ be a class of objects of $\mathscr{A}$. Then the following two conditions are equivalent:
\begin{enumerate}
\item The subprojectivity domain of $\mathcal{L}$ is closed under subobjects.
\item The subprojectivity domain of any object of $\mathcal{L}$ is closed under subobjects.
\end{enumerate}
\end{prop}

\begin{proof}
$ (2)\Rightarrow (1)$ is immediate.

To prove $ (1)\Rightarrow (2)$ let $M\in \mathcal{L}$ and suppose that $B \in {\underline {\mathfrak{Pr}}}_{\mathscr{A}}^{-1}(M)$. Now let $A$ be a subobject of $B$. We get the following pullback diagram $$\xymatrix{0\ar[r]& D \ar[r]\ar[d]_g & P\ar[r]\ar[d]& C \ar[r]\ar@{=}[d]&0 \\ 0\ar[r]& A \ar[r]& B\ar[r]& C\ar[r]&0}$$ where $P \to B$ is an epimorphism and $P$ is projective. Now apply the functor $\Hom(M,-)$ to the previous diagram getting the following commutative diagram with exact rows $$\xymatrix{0\ar[r]& \Hom (M,D) \ar[r]\ar[d]_{\Hom(M,g)}&\Hom (M,P)\ar[r]\ar[d]&\Hom (M,C)\ar@{=}[d] \\ 0\ar[r]& \Hom (M,A) \ar[r]&\Hom (M,B)\ar[r]&\Hom (M,C) }$$

Since $B \in {{\underline {\mathfrak{Pr}}}_{\mathscr{A}}}^{-1}(M)$, we conclude that $\Hom(M,g)$ is epic. Now,  being a projective module, $P\in  {\underline {\mathfrak{Pr}}}_{\mathscr{A}}^{-1}(\mathcal{L})$, then, by $(1)$, $D\in  {\underline {\mathfrak{Pr}}}_{\mathscr{A}}^{-1}(\mathcal{L})$.  Therefore, using Proposition \ref{Lemma-subproj}, we get that $A \in {{\underline {\mathfrak{Pr}}}_{\mathscr{A}}}^{-1}(M)$. 
\end{proof} 

As a consequence, we get the following result, established first in \cite[Proposition 2.4]{Dur2}.
	
\begin{cor}\label{cor-wdim-fp}
Let $R$ be a ring. Then $\wdim R \leq 1$ if and only if the subprojectivity domain of each finitely presented module is closed under submodules.
\end{cor}
\begin{proof} We know that $\wdim R \leq 1$ if and only if flat modules are closed under submodules. But the subprojectivity domain of the class of finitely presented modules is precisely the class of flat modules. Then, we just have to apply Proposition \ref{clsub}.
\end{proof}

Recall that $R$ is left coherent if and only if the category $\mathcal{FG}$ of finitely generated (left) modules is abelian. Then, letting both $\mathcal{L}$ and $\mathscr A$ be the abelian category of finitely generated modules and applying Proposition \ref{clsub} we get the following.

\begin{cor} \label{cor-R-semihered}
Let $R$ be a left coherent ring. Then, $R$ is left semihereditary if and only if the subprojectivity domain of each finitely generated module is closed under submodules.
\end{cor}  
\begin{proof} The result follows from the fact that $Proj(\mathcal{FG})$ is precisely the class of finitely generated projective modules. 
\end{proof}

In \cite[Proposition 2.14]{Sergio} it is studied when the subprojectivity domain of any module is closed under arbitrary direct products. This can be extended to the categorical setting provided (of course) that $\mathscr{A}$ has direct products. 

\begin{prop}\label{clprod}
Suppose that $\mathscr{A}$ has direct products and let $\mathcal{L}$ be a class of objects of $\mathscr{A}$. Then the following conditions are equivalent:
\begin{enumerate}
\item The subprojectivity domain of $\mathcal{L}$ is closed under arbitrary direct products.
\item The subprojectivity domain of any object of  $\mathcal{L}$ is closed under arbitrary direct products.
\end{enumerate}
\end{prop}

\begin{proof}
$(2) \Rightarrow (1)$ is immediate.

For $(1) \Rightarrow (2)$ let $M$ be an object of $\mathcal{L}$, $\{ N_i \}_{i\in I}$ be a family of objects in ${{\underline {\mathfrak{Pr}}}_{\mathscr{A}}}^{-1}(M)$ and $\{g_i : P_i \to N_i \}_{i\in I}$ be a family of epimorphisms where each $P_i$ is projective. Consider the following commutative diagram $$\xymatrix{\Hom(M,\prod_{i\in I}P_i)  \ar[d]_{\psi^P}  \ar[rr]^{\Hom(M, \prod_{i\in I}g_i)} && \Hom(M, \prod_{i\in I}N_i) \ar[d]^{\psi^N} \\ \prod_{i\in I}\Hom( M, P_i) \ar[rr]^{\prod_{i\in I}{\Hom(M, g_i)}} && \prod_{i\in I}\Hom( M, N_i)}$$ where ${\psi^N}$ and ${\psi^P}$ are the natural isomorphisms. The commutativity of the above diagram gives that $\Hom(M, \prod_{i\in I}g_i)$ is epic. Since each $P_i$ is in ${{\underline {\mathfrak{Pr}}}_{\mathscr{A}}}^{-1}(\mathcal{L})$,  $\prod_{i\in I}P_i$ is by assumption in ${{\underline {\mathfrak{Pr}}}_{\mathscr{A}}}^{-1}(\mathcal{L})$. Then $\prod_{i\in I}P_i$ is in ${{\underline {\mathfrak{Pr}}}_{\mathscr{A}}}^{-1}(M)$. By Proposition \ref{Lemma-subproj}, $\prod_{i\in I}N_i \in{{\underline {\mathfrak{Pr}}}_{\mathscr{A}}}^{-1}(M)$ as desired.
\end{proof}

The result \cite[Proposition 2.14]{Sergio} shows that  a ring $R$ is a right perfect and left coherent ring if and only if the subprojectivity domain of any right module is closed under arbitrary direct products. This holds since ${\underline {\mathfrak{Pr}}}_{R-\Mod}^{-1}(R\mbox{-\Mod})$ is the class of projective modules. Here we can give a much direct proof of a characterization of coherent rings given by Durğun in \cite[Proposition 2.3]{Dur2} using also the same property applied to a different class.

\begin{cor}\label{cor-R-coherent}
Let $R$ be a ring. Then $R$ is right coherent if and only if the subprojectivity domain of any finitely presented left module is closed under direct products.
\end{cor}

Given a class of finitely generated modules $\mathcal{S}$, the authors in \cite{Rada-Parra} define $\mathcal{S}$-proj to be the class of modules $N$ such that every morphism $f:S\to N$, where $S\in \mathcal{S}$, factors through a free module. Using Proposition \ref{prop-facto} one can easily see that $\mathcal{S}$-proj is precisely ${{\underline{\mathfrak{Pr}}}_{R-\Mod}}^{-1}(\mathcal{S})$. In \cite[Theorem 3.1]{Rada-Parra}, the authors proved that for any class of finitely generated modules $\mathcal{S}$, ${{\underline {\mathfrak{Pr}}}_{R-\Mod}}^{-1}(\mathcal{S})$ is closed under direct products if and only if any module of $\mathcal{S}$ admits a ${{\underline {\mathfrak{Pr}}}_{R-\Mod}}^{-1}(\mathcal{S})$-preenvelope, if and only if any module of $\mathcal{S}$ admits a projective preenvelope. Now we want to prove this fact for any class of objects.

Recall that a class $\mathcal{F}$ of objects is locally initially small if for every object $M$ of $\A$ there exists a set $\mathcal{F}_M\subseteq \mathcal{F}$ such that every morphism $M\to F$, where $F\in \mathcal{F}$, factors through a direct product of modules in $\mathcal{F}_M$ (see \cite{RS98}). In \cite[Proposition 2.9]{RS98} it is established that the class of projective modules is always locally initially small. The argument consists in proving that for any set $X$, the class $Summ(X)$ is locally initially small, and following the arguments given in \cite{RS98} it is easy to see that this holds in any Grothendieck category with enough projectives. There is only a step of the proof which could deserve a special treatment, namely when the authors prove that the equivalence class $[i]$ is finite. We now give a proof of this fact. 

\begin{prop}\label{fg-finite-set} 
Suppose that $\mathcal{A}$ is a cocomplete abelian category with exact direct limits. Let $X$ be a finitely generated object of $\mathcal{A}$ and $\{A_i\}_{i\in I}$ be a family of objects of $\mathcal{A}$. Then, for every morphism $g:X\to \oplus_{i\in I}A_i$, the set $\{i\in I/ \pi_i g\neq 0\}$ is finite, where $\pi_j:\oplus_{i\in I} A_i\to A_j$ is the canonical projection for every $j\in I$.
\end{prop}
\begin{proof}
Let $M=\Im\ g$ and consider the epic-monic decomposition of $g$, $$\xymatrix{X \ar[rr]^g \ar[rd]_{\bar{g}} & & \oplus_{i\in I}A_i \\ & M \ar[ru]_{f}}$$

If we let $\mathscr{F}$ be the set of all finite subsets of $I$ and define, for every $F\in \mathscr{F}$, $A_F$ to be the image in $\oplus_{i\in I} A_i$ of $\oplus_{i\in F} A_i$, condition (5) of \cite[Chapter 4, Theorem 4.6]{Popescu} says that $M=\sum_{F\in \mathscr{F}}(M\cap A_F)$. 

Now, for every $F\in \mathscr{F}$, we define $(M\cap A_F,\gamma_F,\eta_F)$ as the pullback of $\alpha_F:A_F \to \oplus_{i\in I} A_i$ and $f:M\to \oplus_{i\in I} A_i$. Then, for every two subobjects $(F,\alpha_F) $ and $(F',\alpha_{F'})$ of $\mathscr{F}$ such that $F\subset F'$ (so the diagram $$\xymatrix{ & A_F \ar[dl]_{\alpha_F^{F'}}\ar[d]^{\alpha_F} \\ A_{F'}\ar[r]_{\alpha_{F'}}&  \oplus_{i\in I} A_i \\ }$$ commutes), the universal property of the pullback diagrams guarantees the existence of a family of morphisms $\beta_F^{F'}$ such that the diagram $$\xymatrix{M\cap A_F \ar[rr]^{\gamma_{F}}  \ar@{-->}[dr]^{\beta_F^{F'}} \ar[ddr]_{\eta_{F}} && A_F \ar[ddr]_{\alpha_F}\ar[dr]^{\alpha_F^{F'}} \\ 		& 	M\cap A_{F'} \ar[rr]^{\gamma_{F'}} \ar[d]^{\eta_{F'}} && A_{F'} \ar[d]^{\alpha_{F'}} \\ & M \ar[rr]_f && \oplus_{i\in I} A_i \\ }$$ commutes. Therefore, the family $\{(M\cap A_F)_{F\in \mathscr{F}}, (\beta_F^{F'})_{F\subset F'}\}$ is a direct system of subobjects of $\oplus_{i\in I} A_i$. 

But $M$ is finitely generated so there exists $F_0\in \mathscr{F}$ such that $M=M\cap A_{F_0}$ and then, ${\eta_{F_0}}$ is an isomorphism. Let $j\in I\setminus F_0$. We have $\pi_jf{\eta_{F_0}}=\pi_j\alpha_{F_0}\gamma_{F_0}=0$. Hence $\pi_jf=0$, for every $j\in I\setminus F_0$. Then $\{i\in I/ \pi_i f\neq 0\}\subset F_0$. Therefore the set $\{i\in I/ \pi_i f\neq 0\}$ is finite. Since $g=f\bar{g}$ and $\bar{g}$ epic, the set $\{i\in I/ \pi_i g\neq 0\}$ is also finite.
\end{proof}

\begin{cor} \label{corclequiv}
Suppose that $\mathscr A$ is a locally finitely generated cocomplete abelian category with exact direct limits. Let $\{A_i\}_{i\in I}$ be a family of objects and $f:M\to \oplus_{k\in I}A_k$ be a morphism such  that $\pi_i f \neq 0$ for any $i \in I$. Then  for any $i \in I$, the set $[i]:=\{j\in I/ \pi_i f = \pi_j f \}$ is finite, where $\pi_j:\oplus_{k\in I} A_k\to A_j$ is the canonical projection for every $j\in I$.
\end{cor}
\begin{proof}
Since $\mathscr A$ is locally finitely generated, there exists an epimorphism $g: \oplus_{\alpha \in F} X_\alpha \to M$ with all $X_\alpha$ finitely generated objects.

Call $k_\alpha: X_\alpha \to \oplus_{i \in F} X_i$ the structural monomorphism for any $\alpha \in F$. We claim that for any $j \in I$ there exists some $\alpha_j \in F$ such that $\pi_j fg k_{\alpha_j} \neq 0$. Indeed, if $\pi_j fg k_{\alpha} = 0$ for every $\alpha \in F$ then $\pi_j fg = 0$, and  since $g$ is epic, $\pi_j f = 0$, a contradiction. Therefore, $[j] \subseteq \{i\in I/ \pi_i fg k_{\alpha_j} \neq 0  \}$, and Proposition \ref{fg-finite-set} says that this set is finite.
\end{proof}

As mentioned before, with Corollary \ref{corclequiv}, the proof of \cite[Proposition 2.9]{RS98} follows in any locally finitely generated Grothendieck category with enough projectives $\A$, and as a consequence, if  $\A$ has a system of projective generators ${\cal G}$, then we see that the class $Proj(\A)$ is locally initially small. Indeed, it is easy to check that  $Proj (\A)=Summ(Sum({\cal G}))$. But $Sum({\cal G})$ is a locally initially  small class, so $Summ(Sum({\cal G}))$ is locally initially small too.

With the use of this fact, we can prove the following.

\begin{prop}\label{prop-L-prec-proj}
 Let $\mathcal{L}$ be a class of objects of $\mathscr A$. Then, the following conditions are equivalent:

\begin{enumerate}
	\item Every object $M$ of $\mathcal{L}$ has a projective preenvelope.
	\item Every object $M$ of $\mathcal{L}$ has a ${{\underline {\mathfrak{Pr}}}_{\mathscr{A}}}^{-1}(\mathcal{L})$-preenvelope.
\end{enumerate}

If, in addition, $\A$ is a locally finitely generated Grothendieck category with a system of projective generators then 1. and 2. above are equivalent to
\begin{itemize}
	\item[3.] ${{\underline {\mathfrak{Pr}}}_{\mathscr{A}}}^{-1}(\mathcal{L})$ is closed under direct products.
\end{itemize}

\end{prop}
\begin{proof}
$(1)\Rightarrow (2)$. Let $M\to Q$ be a projective preenvelope of $M$. Let us prove that $M\to Q$ is a ${{\underline {\mathfrak{Pr}}}_{\mathscr{A}}}^{-1}(\mathcal{L})$-preenvelope. For let $N \in {{\underline {\mathfrak{Pr}}}_{\mathscr{A}}}^{-1}(\mathcal{L})$ and let $P \to N$ be an epimorphism with $P$ projective. Apply the functors $\Hom(M,-)$ and $\Hom(Q,-)$ to $P\to N$ to get the following commutative diagram with exact rows $$\xymatrix{\Hom(Q,P)\ar[r]\ar[d]& \Hom(Q,N)\ar[d] \ar[r] & 0 \\ \Hom(M,P) \ar[r] & \Hom(M,N) \ar[r] & 0}$$ with $\Hom(Q,P)\to \Hom(M,P)$ an epimorphism ($M\to Q$ is a projective preenvelope). Therefore, $\Hom(Q,N)\to \Hom(M,N)$ is also an epimorphism.
	
$(2)\Rightarrow (1)$. Let $f: M\to N$ be a ${{\underline {\mathfrak{Pr}}}_{\mathscr{A}}}^{-1}(\mathcal{L})$-preenvelope of $M$ and $g:P \to N$ be an epimorphism with $P$ projective. Since $N \in {{\underline {\mathfrak{Pr}}}_{\mathscr{A}}}^{-1}(\mathcal{L})$ there exists a morphism $h:M\to P$ such that $f=gh$. Let us prove that $h:M\to P$ is a projective preenvelope. 

For let $h': M\to P'$ be a morphism with $P'$ projective. Since  $f:M\to N$ is a ${{\underline {\mathfrak{Pr}}}_{\mathscr{A}}}^{-1}(\mathcal{L})$-preenvelope there exists a morphism $g': N \to P'$ such that $h'= g' f$. Hence $h'= g' g h $ and $g'g:P\to P'$ is the morphism we were looking for.

$(2)\Rightarrow (3)$. Use the same arguments of \cite[Theorem 3.1]{Rada-Parra}.

$(3)\Rightarrow (2)$. Let $M \in \mathcal{L}$. Given $N \in {{\underline {\mathfrak{Pr}}}_{\mathscr{A}}}^{-1}(\mathcal{L})$ we fix an epimorphism $P \to N$ from a projective $P$. Since the class of projective objects is locally initially small, there exists a set $\mathcal X$ of projective objects such that any morphism $M \to P$ factors through a product of objects in the set $\mathcal X$. But every morphism $M\to N$ factors through $P$, and such factorization $M\to P$ factors through a product of objects in the set $\mathcal X$, so we have just seen that every morphism $M \to N$ with $N \in {{\underline {\mathfrak{Pr}}}_{\mathscr{A}}}^{-1}(\mathcal{L})$ factors through a product of elements of $\mathcal{X}$. 

Call now $K= \prod_{P \in \mathcal{X}} P ^{\Hom(M,P)}$. Since ${{\underline {\mathfrak{Pr}}}_{\mathscr{A}}}^{-1}(\mathcal{L})$ is supposed to be closed under direct products, we see that $K \in {{\underline {\mathfrak{Pr}}}_{\mathscr{A}}}^{-1}(\mathcal{L})$. 

Now, for each $P \in \mathcal{X}$ there exists a canonical morphism $\lambda_P: M \to P ^{\Hom(M,P)}$, so there is a unique $\lambda: M\to K$ such that $\pi_P \lambda= \lambda_P$ for every $P\in \mathcal{X}$, where  $\pi_P$ are the canonical projections. We claim that $\lambda: M\to K$ is a ${{\underline {\mathfrak{Pr}}}_{\mathscr{A}}}^{-1}(\mathcal{L})$-preenvelope of $M$. 

To show this, take any morphism  $f:M \to N$ with $N \in {{\underline {\mathfrak{Pr}}}_{\mathscr{A}}}^{-1}(\mathcal{L})$, so there exist $h: M \to \prod_{X \in \mathcal{X}}X$ and $g: \prod_{X \in \mathcal{X}}X \to N$ such that $f=gh$. Consider the projections $\pi_X : K \to X^{\Hom(M,X)}$ and $\pi_{p_X  h} : X^{Hom(M,X)} \to X$ (the projection to the component $p_X  h$ where $p_X : \prod_{X \in \mathcal{X}}X \to X$ the canonical projection). By the universal property of the direct product there exists a unique morphism $\gamma: K\to \prod_{X \in \mathcal{X}}X $ such that $p_X \gamma=  \pi_{p_X h} \pi_X$. Therefore, $p_X \gamma \lambda= \pi_{p_X  h} \pi_X \lambda= \pi_{p_X h} \lambda_X= p_X h$ for all $X\in \mathcal{X}$, so $\gamma \lambda=h$ and hence $g \gamma \lambda=gh=f$. We then get that $\lambda: M\to K$ is a ${{\underline {\mathfrak{Pr}}}_{\mathscr{A}}}^{-1}(\mathcal{L})$-preenvelope of $M$.
\end{proof} 

It is a natural question at this point to ask about the closure of subprojectivity domains under arbitrary direct sums. There is not a clear answer to this. In \cite[Proposition 2.13]{Sergio} it is shown that the subprojectivity domain of any finitely generated module is closed under arbitrary direct sums. Now, we will see that the class for which subprojectivity domains are closed under direct sums is larger than that of finitely generated modules, since it contains that of small modules. Whether or not this is the largest class with this property we don't know, but it would be of a great interest to know to what point this class can be enlarged.

So suppose that $\mathscr{A}$ is an abelian category with direct sums and let $M$ be an object in $\mathscr{A}$, $\{N_i\}_{i\in I}$ be a family of objects in ${{\underline {\mathfrak{Pr}}}_{\mathscr{A}}}^{-1}(M)$ and consider a family of epimorphisms $P_i\to N_i$, where each $P_i$ is projective. We have the following commutative diagram $$\xymatrix{0\ar[r ]& \oplus_{i\in I}\Hom(M,P_i) \ar[r]^{\phi_P}\ar[d]_\alpha&\Hom (M,\oplus_{i\in I}P_i))\ar[r]\ar[d]_\beta & \Coker{\phi_P}\ar[r]\ar[d]_\gamma & 0 \\ 0\ar[r]& \oplus_{i\in I}\Hom(M,N_i) \ar[r]_{\phi_N}&\Hom (M,\oplus_{i\in I}N_i))\ar[r]&\Coker{\phi_N}\ar[r]&0}$$

Clearly, $\beta$ is epic if and only if $\gamma$ is epic since $\alpha$ is an epimorphism. Consequently,  if $M$ is a small object, that is, $\Hom(M,-)$ preserves direct sums, then the subprojectivity domain of $M$ is closed under direct sums.


\section{Subprojectivity and the Ext-orthogonal classes} \label{sec4}

Throughout this section we will suppose $\mathcal{A}$ to have enough injectives.

The aim of this section is to establish the relation between the subprojectivity domains and the Ext-orthogonal classes. The idea behind this result is inspired by the following discussion. Fix a class $\mathcal{L}$ and consider a short exact sequence $0 \to K \to P \to N\to 0$ with $P$ is projective. For every $M \in \mathcal{L}$, we have the exact sequence $$\cdots \to \Hom(M,P) \to \Hom(M,N) \to  \Ext^{1}(M,K) \to  \Ext^{1}(M,P) \to  \cdots.$$ So if we assume that $\mathcal{\mathcal{L}^\perp}$ contains all projective objects, we get the following equivalence: $N \in {{\underline {\mathfrak{Pr}}}_{\mathscr{A}}}^{-1}(\mathcal{L})$ if and only if $K \in \mathcal{\mathcal{L}^\perp}$. However, it does not seem clear how to get new results if  we relate the subprojectivity domains with a property on   kernels of epimorphisms. But, if we suppose moreover that  $\mathcal{L}$ contains all projective objects and that it is closed under kernels of epimorphisms, then $\mathcal{\mathcal{L}^\perp}$ will be closed under cokernels of monomorphisms (see  \cite[Lemma 1.2.8]{Ramon}). So by the above equivalence we get the following implication: if $N \in {{\underline {\mathfrak{Pr}}}_{\mathscr{A}}}^{-1}(\mathcal{L})$ then $N \in \mathcal{\mathcal{L}^\perp}$, that is, ${{\underline {\mathfrak{Pr}}}_{\mathscr{A}}}^{-1}(\mathcal{L}) \subseteq  \mathcal{\mathcal{L}^\perp}$. In the following result we provide a necessary and sufficient condition to have the equality ${{\underline {\mathfrak{Pr}}}_{\mathscr{A}}}^{-1}(\mathcal{L}) = \mathcal{\mathcal{L}^\perp}$. 

\begin{thm}\label{thm-main4}
Let $\mathcal{L}$ be a class of objects of $\mathscr A$ which is closed under kernels of epimorphisms and which contains the class  $Proj(\mathscr{A})$. Then, the following conditions are equivalent:
\begin{enumerate}
\item $ \mathcal{\mathcal{L}^\perp} ={{\underline {\mathfrak{Pr}}}_{\mathscr{A}}}^{-1}(\mathcal{L})$.
\item $ \mathcal{L} \bigcap \mathcal{\mathcal{L}^\perp} = Proj(\mathscr A)$ and every object in $\mathcal{L}^\perp$ has a special $\mathcal{L}$-precover.
\end{enumerate}

If, in addition, $\A$ is a locally finitely generated Grothendieck category with a system of projective generators then 1. and 2. above are equivalent to
\begin{itemize}
\item[3.] $Proj(\mathcal{A}) \subseteq \mathcal{L}^\perp$, ${{\underline {\mathfrak{Pr}}}_{\mathcal{A}}}^{-1}(\mathcal{L})$ is closed under cokernels of monomorphisms  and every $M \in \mathcal{L}$ has an $\mathcal{L}^\perp$-preenvelope which is projective.
\end{itemize} 
\end{thm}
\begin{proof}
$1.\Rightarrow 2.$ Let us prove first that $\mathcal{L}\bigcap \mathcal{\mathcal{L}^\perp} \subseteq  Proj(\mathscr A).$ If $M\in \mathcal{L}\bigcap \mathcal{\mathcal{L}^\perp}$ then $M \in {{\underline {\mathfrak{Pr}}}_{\mathscr{A}}}^{-1}(\mathcal{L})$ by condition 1., and then $M \in {{\underline {\mathfrak{Pr}}}_{\mathscr{A}}}^{-1}(M)$. Hence, $M$ is projective by Proposition \ref{prop-whole-cat-obj}.
 
Conversely, any projective $P$ holds in $\mathcal{L}$ by the hypotheses, and of course $P \in {{\underline {\mathfrak{Pr}}}_{\mathscr{A}}}^{-1}(\mathcal{L})$. But ${{\underline {\mathfrak{Pr}}}_{\mathscr{A}}}^{-1}(\mathcal{L})= \mathcal{\mathcal{L}^\perp}$, so indeed $P\in \mathcal{L} \bigcap \mathcal{\mathcal{L}^\perp}$. 

To prove the second assertion let $N \in \mathcal{L}^\perp$ and let us show that any epimorphism $g:P \to N$ with $P$ projective is indeed a special $\mathcal{L}$-precover. 

That $g$ is an $\mathcal{L}$-precover is clear since, by assumption, $P \in \mathcal{L}$ and $N\in{{\underline {\mathfrak{Pr}}}_{\mathscr{A}}}^{-1}(\mathcal{L})$. 

Now,  for an object $L\in \mathcal{L}$, being $\Hom_{\A}(L,P)\to \Hom_{\A} (L,N)\to 0$ exact  and being $\Ext^1_{\A} (L,P)=0$ implies that $\Ext^1_{\A} (L,\ker g)=0$, that is, $\ker g\in \mathcal{L}^\perp$.

$2.\Rightarrow 1.$ Let $N\in \mathcal{L}^\perp$, $M\in \mathcal{L}$ and consider a special $\mathcal{L}$-precover $$0\to K\to L\to N\to 0$$ of $N$.

Since $N$ and $K$ are in $\mathcal{L}^\perp$, $L$ do too and then we get $L \in \mathcal{L}\bigcap \mathcal{\mathcal{L}^\perp}$. But $\mathcal{L} \bigcap \mathcal{\mathcal{L}^\perp}=Proj(\mathscr A)$ so by  Proposition \ref{Lemma-subproj}, $N \in {{\underline {\mathfrak{Pr}}}_{\mathscr{A}}}^{-1}(M)$.
	
Conversely, let $N\in {{\underline {\mathfrak{Pr}}}_{\mathscr{A}}}^{-1}(\mathcal{L})$ and let $$0\to K \to P \to N \to 0$$ be a short exact sequence with $P$ projective. For any $L \in \mathcal{L}$ the associated long exact sequence looks like $$\cdots \to \Ext^1(L,P) \to \Ext^1(L,N) \to \Ext^2(L,K) \to \cdots.$$ But $\Ext^1(L,P)=0$ since $P$ is projective (so $P\in \mathcal{L}^{\perp}$ by the hypothesis), so proving that $\Ext^2(L,K)=0$ will give $N\in \mathcal{L}^{\perp}$.

Let $0\to C \to Q \to L \to 0$ be a short exact sequence with $Q$ projective. Since $\mathcal{L}$ is closed under kernels of epimorphisms, $Q,\ L\in \mathcal{L}$ implies $C\in \mathcal{L}$. Now, in the long exact sequence $$\cdots \to \Ext^1(C,K) \to \Ext^2(L,K) \to \Ext^2(Q,K) \to \cdots$$ we have $\Ext^1(C,K)=0$. Indeed, in the long exact sequence 
$$\Hom(C,P) \to \Hom(C,N) \to \Ext^1(C,K)  \to \Ext^1(C,P)   \to \cdots$$
$\Ext^1(C,P)=0$ (since  $P\in \mathcal{L}^{\perp}$ by the hypothesis) and the first morphism is epic (since $N\in {{\underline {\mathfrak{Pr}}}_{\mathscr{A}}}^{-1}(\mathcal{L})$ and  $C\in \mathcal{L}$).   On the other hand, $\Ext^2(Q,K)=0$, so indeed $\Ext^2(L,K)=0$.

$1.\Rightarrow 3.$ Clearly $Proj(\mathcal{A}) \subseteq {{\underline {\mathfrak{Pr}}}_{\mathcal{A}}}^{-1}(\mathcal{L}) =\mathcal{L}^\perp$. 

Now, let $$0\to A\to B\to C\to 0$$ be exact with $A,\ B\in {{\underline {\mathfrak{Pr}}}_{\mathcal{A}}}^{-1}(\mathcal{L})$. To prove that $C\in {{\underline {\mathfrak{Pr}}}_{\mathcal{A}}}^{-1}(\mathcal{L}) $ choose any $L\in \mathcal{L}$ and apply the functor $\Hom_{\A}(L,-)$ to the exact sequence. We get a long exact sequence $$0 \to \Hom (L,A) \to \Hom (L,B) \to \Hom (L,C) \to \Ext^1(L,A)$$ with $\Ext^1(L,A)=0$ since $L\in\mathcal{L}$ and $A\in {{\underline {\mathfrak{Pr}}}_{\mathcal{A}}}^{-1}(\mathcal{L})=\mathcal{L}^{\perp}$. Then Proposition \ref{Lemma-subproj} immediately gives that $C\in {{\underline {\mathfrak{Pr}}}_{\mathcal{A}}}^{-1}(\mathcal{L})$.

Finally, if $M \in \mathcal{L}$ let $f: M \to N$ be a ${{\underline {\mathfrak{Pr}}}_{\mathcal{A}}}^{-1}(\mathcal{L})$-preenvelope, which exists by Proposition \ref{prop-L-prec-proj} since $\mathcal{L}^\perp$ is always closed under direct products.

Now, find an epimorphism $g:P \to N$ from a projective $P$. Since $N \in {{\underline {\mathfrak{Pr}}}_{\mathcal{A}}}^{-1}(\mathcal{L})$ there exists a morphism $h: M \to P$ such that $f=gh$. We claim that $h: M \to P$ is a ${{\underline {\mathfrak{Pr}}}_{\mathcal{A}}}^{-1}(\mathcal{L})$-preenvelope. Indeed, let $k: M \to N'$ be a morphism with $N' \in {{\underline {\mathfrak{Pr}}}_{\mathcal{A}}}^{-1}(\mathcal{L})$. Since $f: M \to N$ is a  ${{\underline {\mathfrak{Pr}}}_{\mathcal{A}}}^{-1}(\mathcal{L})$-preenvelope,  there exists $l: N \to N'$ such that $k=lf$, hence $k=lgh$. Therefore, $h: M \to P$ is a ${{\underline {\mathfrak{Pr}}}_{\mathcal{A}}}^{-1}(\mathcal{L})$-preenvelope.

$3.\Rightarrow 1.$ Let $N \in \mathcal{L}^\perp$, choose any $M \in \mathcal{L}$, any morphism $f: M \to N$ and a $\mathcal{L}^\perp$-preenvelope $g:M \to Q$ of $M$ with $Q$ projective. Then there exists $h:Q \to N$ such that $f=hg$, so $N \in {{\underline {\mathfrak{Pr}}}_{\mathcal{A}}}^{-1}(\mathcal{L})$ by Proposition \ref{prop-facto}. 

Conversely, let  $N \in {{\underline {\mathfrak{Pr}}}_{\mathcal{A}}}^{-1}(\mathcal{L})$, choose any $M \in \mathcal{L}$ and take its $\mathcal{L}^\perp$-preenvelope $g:M \to Q$, where $Q$ is projective. Of course every $\mathcal{L}^\perp$-preenvelope is injective since $\mathcal{L}^\perp$ contains the class of injectives, so if $C$ is the cokernel of $g$ we get a long exact sequence $$\xymatrix{\cdots\ar[r]& \Ext^1(Q,N) \ar[r]&  \Ext^1(M,N) \ar[r]&  \Ext^2(C,N) \ar[r]& \cdots}$$

Since $Q$ is projective, showing that $\Ext^2(C,N)=0$ would immediately give $\Ext^1(M,N)=0$, so let's prove that $\Ext^i(C,N)=0,\ i=1,2$.

Choose then any morphism $f:M \to N$ and find an epimorphism $h:P\to N$ from a projective $P$. Then, by the $N$-subprojectivity of $M$, then diagram $$\xymatrix{ & M \ar[d]^f \ar@{-->}[dl]_k \ar[r]^g & Q \\ P \ar[r]_h & N \ar[r] & 0}$$ can be completed commutatively by $k$. But $P$ is projective (so it holds in $\mathcal{L}^{\perp}$ by the hypotheses), and $g$ is a $\mathcal{L}^\perp$-preenvelope, so there is a morphism $l: Q \to P$ such that $lg=k$. Therefore, $f=hk=hlg$ and then $\Hom(g,N)$ is an epimorphism, so from the long exact sequence $$\xymatrix{\cdots\ar[r]& \Hom(Q,N) \ar[r]& \Hom(M,N) \ar[r]&  \Ext^1(C,N) \ar[r]&  \Ext^1(Q,N)=0 \ar[r]&\cdots}$$ we see that $ \Ext^1(C,N)=0$. 

Now, if $0 \to N \to E \to D \to 0$ is exact and $E$ is injective, we get an associated long exact sequence $$\xymatrix{\cdots\ar[r]& \Ext^1(C,D) \ar[r]&  \Ext^2(C,N) \ar[r]&  \Ext^2(C,E)=0 \ar[r]&\cdots}$$ But we have proved already that $\mathcal{L}^\perp \subseteq {{\underline {\mathfrak{Pr}}}_{\mathcal{A}}}^{-1}(\mathcal{L})$, so $E \in {{\underline {\mathfrak{Pr}}}_{\mathcal{A}}}^{-1}(\mathcal{L})$, and $N$ does too, so since ${{\underline {\mathfrak{Pr}}}_{\mathcal{A}}}^{-1}(\mathcal{L})$ is closed under cokernels of monomorphisms we get that  $D$ is also in ${{\underline {\mathfrak{Pr}}}_{\mathcal{A}}}^{-1}(\mathcal{L})$. Hence, by the same arguments as before, $\Ext^1(C,D)=0$ and then $\Ext^2(C,N)=0$.
\end{proof}

Theorem \ref{thm-main4} can be used to characterize the subprojectivity domain of the class of exact complexes  $\mathcal{E}$. We recall that every projective complex is exact, that the class $\mathcal{E}$ is closed under kernels of epimorphisms and that it is special precovering in the whole category of complexes (see \cite[Theorem 2.3.17]{Ramon}). It is also known that $\mathcal{E} \bigcap \mathcal{E}^\perp$ is the class of injective complexes (see \cite[Proposition 2.3.7]{Ramon}). So, by Theorem \ref{thm-main4}, we get the following result. 

\begin{cor}\label{cor-R-quasifrob}
The subprojectivity domain of the class of exact complexes of modules is $\mathcal{E}^\perp$ if and only if  $R$ is quasi-Fr\"obenius.
\end{cor}

The question of whether or not any object has a special $\mathcal{GP}$-precover has been a subject of many papers. Here, as a consequence of Corollary \ref{prop-subproj-goren-cl} and Theorem \ref{thm-main4} (since it is known that the class $\mathcal{GP}$ is closed under kernels of epimorphisms) we immediately get a partial answer which has been recently known following different methods (see \cite[Proposition 4.1]{ZH19}).

\begin{cor}\label{cor-specialGP-precover}
If direct sums exist and they are exact then, every object in $\mathcal{GP}^\perp$ has a special $\mathcal{GP}$-precover.
\end{cor}

We end this paper with the following remark. One can see that the proof of the equivalence $(1) \Leftrightarrow (3)$ in Theorem \ref{thm-main4} does not need the class $\mathcal{L}$ to be closed under kernels of epimorphisms since we can find an example of such a situation: we know that the class of strongly Gorenstein projective objects is not closed under kernels of epimorphisms and by Proposition \ref{prop-subproj-sgoren-cl}, we have that the subprojectivity domain of $\mathcal{SGP}$ is closed under cokernels of monomorphisms. \bigskip

\noindent\textbf{Acknowledgment.}  The authors want to express their gratitude to the referee for the very helpful comments and suggestions. Part of this work was done while Amzil and Ouberka were visiting the University of Almer\'{i}a under the Erasmus+ KA107 scholarship grant from November 2019 to July 2020.  They would like to express their sincere thanks for the warm hospitality and the excellent working conditions.\\
Also, during the preparation of this paper,  Bennis has been invited by the University of Almer\'{i}a. He is always thankful to the department of Mathematics  for the warm hospitality during his stay in Almer\'{i}a.\\
The  authors Luis Oyonarte and J.R. Garc\'{\i}a Rozas were partially supported by Ministerio de Econom\'{\i}a y Competitividad, grant reference 2017MTM2017-86987-P.
\bigskip


Houda Amzil:   CeReMAR Center; Faculty of Sciences, Mohammed V University in Rabat, Rabat, Morocco.

\noindent e-mail address: houda.amzil@um5r.ac.ma; ha015@inlumine.ual.es
		
 Driss Bennis:   CeReMAR Center; Faculty of Sciences, Mohammed V University in Rabat, Rabat, Morocco.

\noindent e-mail address: driss.bennis@um5.ac.ma; driss$\_$bennis@hotmail.com

J. R. Garc\'{\i}a Rozas: Departamento de  Matem\'{a}ticas,
Universidad de Almer\'{i}a, 04071 Almer\'{i}a, Spain.

\noindent e-mail address: jrgrozas@ual.es

Hanane Ouberka:   CeReMAR Center; Faculty of Sciences, Mohammed V University in Rabat, Rabat, Morocco.

\noindent e-mail address: hanane$\_$ouberka@um5.ac.ma; ho514@inlumine.ual.es

Luis Oyonarte: Departamento de  Matem\'{a}ticas, Universidad de
Almer\'{i}a, 04071 Almer\'{i}a, Spain.

\noindent e-mail address: oyonarte@ual.es


\begin{thebibliography}{999}\addcontentsline{toc}{chapter}{\textbf{Bibliography}}
\bibitem{Alagoz} 
Y. Alagöz and E. Büyükas, \textit{Max-projective modules},   J. Algebra Appl.,    Available online, DOI  2150095  (2021).
\bibitem{BeHuWa15}
D. Bennis, K. Hu, and F. Wang, \textit{On $2$-{SG}-semisimple rings}, Rocky Mountain J. Math.  {\bf 45} (2015), 1093--1100.
\bibitem{BM}
D. Bennis and N. Mahdou, \textit{Strongly Gorenstein projective, injective, and flat modules}, J. Pure Appl. Algebra {\bf 210} (2007), 437--445.
\bibitem{BM2}
D. Bennis and N. Mahbou, \textit{Global Gorenstein dimensions}, Proc. Amer. Math. Soc. {\bf 138} (2010) 461--465.
\bibitem{BMO}
D. Bennis, N. Mahdou and K. Ouarghi, \textit{Rings over which all modules are strongly Gorenstein projective}, Rocky Mountain J. Math. {\bf 40} (2010), 749--759.
\bibitem{Lars} 
L. W. Christensen and H.-B. Foxby, \textit{Hyperhomological Algebra with Applications to Commutative Rings}, preprint (2006).
\bibitem{Lars2}
L.W. Christensen, A. Frankild and H. Holm,  \textit{On Gorenstein projective, injective and flat dimensions—A functorial description with applications},  J. Algebra {\bf  302} (2006), 231--279.
\bibitem{Dur1} 
Y. Durğun, \textit{Rings whose modules have maximal or minimal subprojectivity domain},  J. Algebra Appl.  {\bf 14} (2015), 1550083.
\bibitem{Dur2} 
Y. Durğun, \textit{On subprojectivity domains of pure-projective modules}, J. Algebra Appl. \textbf{19} (2020), 2050091.
\bibitem{EnoEs} 
E. Enochs and S. Estrada, \textit{Projective representations of quivers}, Comm. Algebra   {\bf 33} (2005), 3467--3478.
\bibitem{Enochs} 
E. E. Enochs and O. M. G. Jenda, \textit{Relative homological algebra}, Walter de Gruyter, Berlin-New York (2000).
\bibitem{Fieldhouse}  D. J. Fieldhouse, \text{Characterizations of modules},  Canad. J. Math. {\bf 23}  (1971), 608-610.
\bibitem{Ramon} J. R. Garc\'{\i}a Rozas, \textit{Covers and Envelopes in the Category of Complexes of Modules}, Chapman \& Hall/CRC, Research Notes in Mathematics, \textbf{407}, (1999).
\bibitem{Herzog}   
I. Herzog, \textit{Pure-injective envelopes}, J. Algebra Appl. {\bf 2} (2003), 397--402.
\bibitem{Sergio} 
C. Holston, S. R. López-Permouth, J. Mastromatteo and J. E. Simental-Rodriguez, \textit{An alternative perspective on projectivity of modules}, Glasgow Math. J.  {\bf 57}  (2015), 83--99.
\bibitem{Mao} 
L. Mao, \textit{When does every simple module have a projective envelope?}, Comm. Algebra {\bf 35} (2007), 1505--1516.  
\bibitem{Sang} 
S. Park, \textit{Projective representations of quivers},  Int. J. Math. \& Math. Sci.  \textbf{31} (2002), 97--101.
\bibitem{Rada-Parra} 
R. Parra and J. Rada, \textit{Projective envelopes of finitely generated modules}, Algebra Colloq. \textbf{18} (2011), 801--806.
\bibitem{Popescu} 
N. Popescu and L. Popescu, \textit{Theory of categories}, Editura Academiei Republicii Socialiste România (1979).
\bibitem{RS98}    
J. Rada and M. Saor\'{i}n, \textit{Rings characterized by (pre)envelopes and (pre)covers of their modules},   Comm. Algebra  \textbf{26} (1998),  899--912.
\bibitem{Sten} 
B. Stenström,  \textit{Purity in functor categories}, J. Algebra  {\bf 8}  (1968), 352--361.
\bibitem{Tri} 
J. Trlifaj, \textit{Whitehead test modules}, Trans. Amer. Math. Soc.  {\bf 348} (1996), 1521--1554. 
\bibitem{Warfield} 
R. B. Warfield, \textit{Purity and algebraic compactness for modules}, Pacific J. Math. {\bf 28}  (1969), 699--719.
\bibitem{ZH19}  
T. Zhao and Z. Huang, \textit{Special precovered categories of Gorenstein categories},  Sci. China Ser. A {\bf 62} (2019),  1553--1566.   
\end{thebibliography}
\end{document}